\documentclass{article}

    \usepackage[preprint]{neurips_2023}

\usepackage[utf8]{inputenc} 
\usepackage[T1]{fontenc}    
\usepackage[pagebackref=true]{hyperref}       
\usepackage{url}            
\usepackage{booktabs}       
\usepackage{amsfonts}       
\usepackage{nicefrac}       
\usepackage{microtype}      
\usepackage{xcolor}         
\usepackage{amsmath}
\usepackage{amsthm}
\usepackage{amssymb}
\usepackage{algorithm, algpseudocode}
\usepackage{mathtools}
\usepackage[capitalize,noabbrev]{cleveref}
\usepackage[english]{babel}
\usepackage{natbib}
\usepackage{bbm}
\usepackage{verbatim}

\theoremstyle{plain}
\newtheorem{theorem}{Theorem}[section]
\newtheorem{proposition}[theorem]{Proposition}
\newtheorem{lemma}[theorem]{Lemma}

\theoremstyle{definition}

\newtheorem{definition}[theorem]{Definition}

\theoremstyle{remark}

\renewcommand*\backref[1]{\ifx#1\relax \else Cited on #1 \fi}

\usepackage[textsize=tiny]{todonotes}

\newcommand{\red}[1]{{\color{red} (#1)}}

\DeclareMathOperator*{\argmin}{argmin} 
\DeclareMathOperator*{\argmax}{argmax} 
\newcommand{\R}{\mathbb{R}}

\makeatletter
\newcommand*\bigcdot{\mathpalette\bigcdot@{.5}}
\newcommand*\bigcdot@[2]{\mathbin{\vcenter{\hbox{\scalebox{#2}{$\m@th#1\bullet$}}}}}
\makeatother
\newcommand{\norm}[1]{|| #1 ||}

\newcommand{\innerprod}[2]{\langle #1, #2 \rangle}

\newcommand{\supp}{\text{supp}}
\newcommand{\Null}{\text{Null}}

\newcommand{\parens}[1]{\left(#1\right)}

\newcommand{\curly}[1]{\left\{#1\right\}}

\def\bbE{\mathbb{E}}

\def\calB{\mathcal{B}}

\def\calD{\mathcal{D}}

\def\calS{\mathcal{S}}

\usepackage{mathtools} 

\newcommand{\rbr}[1]{\left(#1\right)}
\newcommand{\sbr}[1]{\left[#1\right]}
\newcommand{\cbr}[1]{\left\{#1\right\}}
\newcommand{\abr}[1]{\left\langle#1\right\rangle}

\def\norm#1{\|#1\|}

\newcommand{\biggmid}{\bigg \vert }

\def\argmax{\mathop{\rm arg\,max}}
\def\argmin{\mathop{\rm arg\,min}}

\def\half{\frac 1 2}

\newcommand{\xbar}{\xbar{w}}

\newcommand{\E}{\bbE}

\title{Analyzing and Improving Greedy 2-Coordinate Updates for Equality-Constrained Optimization via Steepest Descent in the 1-Norm}

\author{
  Amrutha Varshini Ramesh \\
  University of British Columbia \\
  \And
  Aaron Mishkin \\
  Stanford University \\
  \AND
  Mark Schmidt \\
  University of British Columbia,\\ Canada CIFAR AI Chair (Amii)\\
  \And
  Yihan Zhou \\
  University of Texas at Austin\\
  \And
  Jonathan Wilder Lavington \\
  University of British Columbia \\
  \And
  Jennifer She \\
  DeepMind\\
}
\begin{document}
    \maketitle

    \begin{abstract}
We consider minimizing a smooth function subject to a summation constraint over its variables. 
By exploiting a connection between the greedy 2-coordinate update for this problem and equality-constrained steepest descent in the 1-norm, we give a convergence rate for greedy selection under a proximal Polyak-\L{}ojasiewicz assumption that is faster than random selection and independent of the problem dimension $n$. We then consider minimizing with both a summation constraint and bound constraints, as arises in the support vector machine dual problem. Existing greedy rules for this setting either guarantee trivial progress only 
 or require $O(n^2)$ time to compute. We show that bound- and summation-constrained steepest descent in the L1-norm guarantees more progress per iteration than previous rules and can be computed in only $O(n \log n)$ time.
\end{abstract}

\section{Introduction}
Coordinate descent (CD) is an iterative optimization algorithm that performs a gradient descent step on a single variable at each iteration.
CD methods are appealing because they have a convergence rate similar to gradient descent, but for some common objective functions the iterations have a much lower cost.
Thus, there is substantial interest in using CD for training machine learning models.

\noindent\textbf{Unconstrained coordinate descent}: \citet{nesterov2012efficiency} considered CD with random choices of the coordinate to update, and proved non-asymptotic linear convergence rates for strongly-convex functions with Lipschitz-continuous gradients. It was later shown that these linear convergence rates are achieved under a generalization of strong convexity called the Polyak-\L{}ojasiewicz condition~\citep{karimi2016linear}. Moreover, greedy selection of the coordinate to update also leads to faster rates than random selection~\citep{nutini2015coordinate}. These faster rates do not depend directly on the dimensionality of the problem due to an equivalence between the greedy coordinate update and steepest descent on all coordinates in the 1-norm. For a discussion of many other problems where we can implement greedy selection rules at similar cost to random rules, see~\citet[][Sections~2.4-2.5]{nutini2022let}.

\noindent\textbf{Bound-constrained coordinate descent}: CD is commonly used for optimization with lower and/or upper bounds on each variable. \citet{nesterov2012efficiency} showed that the unconstrained rates of randomized CD can be achieved under these separable constraints using a projected-gradient update of the coordinate. \citet{richtarik2014iteration} generalize this result to include a non-smooth but separable term in the objective function via a proximal-gradient update; this justifies using CD in various constrained and non-smooth settings, including least squares regularized by the 1-norm and support vector machines with regularized bias. Similar to the unconstrained case,~\citet{karimireddy2019efficient} show that several forms of greedy coordinate selection lead to faster convergence rates than random selection for problems with bound constraints or separable non-smooth terms.

\noindent\textbf{Equality-constrained coordinate descent}: many problems in machine learning  require us to satisfy an equality constraint. The most common example is that discrete probabilities must sum to one. Another common example is SVMs with an unregularized bias term. The (non-separable) equality constraint cannot be maintained by single-coordinate updates, but it can be maintained if we update two variables at each iteration.~\citet{necoara2011random} analyze random selection of the two coordinates to update, while~\citet{fang2018faster} discuss randomized selection with tighter rates. The LIBSVM package~\citep{chang2011libsvm} uses a greedy 2-coordinate update for fitting SVMs which has the same cost as random selection.
But despite LIBSVM being perhaps the most widely-used CD method of all time, current analyses of greedy 2-coordinate updates either result in sublinear convergence rates or do not lead to faster rates than random selection~\citep{tseng2009block,beck20142}.

\textbf{Our contributions}: we first give a new analysis for the greedy 2-coordinate update for optimizing a smooth function with an equality constraint. The analysis is based on an equivalence between the greedy update and equality-constrained steepest descent in the 1-norm. This leads to a simple dimension-independent analysis of greedy selection showing that it can converge substantially faster than random selection. Next, we consider greedy rules when we have an equality constraint and bound constraints. We argue that the rules used by LIBSVM cannot guarantee non-trivial progress on each step. We analyze a classic greedy rule based on maximizing progress, but this analysis is dimension-dependent and the cost of implementing this rule is $O(n^2)$ if we have both lower and upper bounds. Finally, we show that steepest descent in the 1-norm with equalities and bounds guarantees a fast dimension-independent rate and can be implemented in $O(n \log n)$. This rule may require updating more than 2 variables, in which case the additional variables can only be moved to their bounds, but this can only happen for a finite number of early iterations.

%
%
\section{Equality-Constrained Greedy 2-Coordinate Updates}
We first consider the problem of minimizing a twice-differentiable function $f$ subject to a simple linear equality constraint,
\begin{equation}
\min_{x\in\R^n} f(x), \hspace{0.3 cm}\text{subject to } \sum\limits_{i=1}^n x_i =\gamma,
 \label{eq:BCD}   
\end{equation}
where $n$ is the number of variables and $\gamma$ is a constant. On iteration $k$ the 2-coordinate optimization method chooses a coordinate $i_k$ and a coordinate $j_k$ and updates these two coordinates using 
\[
x_{i_k}^{k+1} = x_{i_k}^k + \delta^k, \quad x_{j_k}^{k+1} = x_{j_k}^k - \delta^k,
\]
for a scalar $\delta^k$ (the other coordinates are unchanged). We can write this update for all coordinates as 
\begin{equation}
\label{eq:dk}
x^{k+1} = x^k + d^k, \quad \text{where $d_{i_k}^k=\delta^k$, $d_{j_k}^k=-\delta^k$, and $d_m^k=0$ for $m \not\in\{i_k,j_j\}$.}
\end{equation}
If the iterate $x^k$ satisfies the constraint then this update maintains the constraint. In the coordinate gradient descent variant of this update we choose $\delta^k = -\frac{\alpha^k}{2}(\nabla_{i_k} f(x^k) - \nabla_{j_k}f(x^k))$ for a step size $\alpha_k$. This results in an update to $i_k$ and $j_k$ of the form
\begin{equation}\label{eq:update2}
\begin{aligned}
x_{i_k}^{k+1} & = x_{i_k}^k -\frac{\alpha^k}{2}(\nabla_{i_k} f(x^k) - \nabla_{j_k}f(x^k)), \, \,
x_{j_k}^{k+1} = x_{j_k}^k -\frac{\alpha^k}{2}(\nabla_{j_k} f(x^k) - \nabla_{i_k}f(x^k)).
\end{aligned}
\end{equation}
If $f$ is continuous, this update is guaranteed to decrease $f$ for sufficiently small $\alpha^k$.
The greedy rule chooses the coordinates to update by maximizing the difference in their partial derivatives,
\begin{equation}
\label{eq:ik}
i_k \in \argmax_i\nabla_i f(x^k), \quad j_k \in \argmin_j\nabla_j f(x^k).
\end{equation}
At the solution of the problem we must have partial derivatives being equal, so intuitively this greedy choice updates the coordinates that are furthest above/below the average partial derivative. 
This choice also minimizes a set of 2-coordinate quadratic approximations to the function (see Appendix~\ref{proof:b_and_d})
\[
\argmin_{i,j}\left\{\min_{d_{ij}|d_i+d_j=0}f(x^k) + \nabla_{ij} f(x^k)^Td_{ij} + \frac{1}{2\alpha^k}\norm{d_{ij}}^2\right\},
\]
which is a special case of the Gauss-Southwell-q (GS-q) rule of~\citet{tseng2009block}.

We assume that the gradient of $f$ is Lipschitz continuous, and our analysis will depend on a quantity we call $L_2$. The quantity $L_2$ bounds the change in the $2$-norm of the gradient with respect to any two coordinates $i$ and $j$ under a two-coordinate update of any $x$ of the form~\eqref{eq:dk}.
\begin{equation}
\label{eq:L2-smooth}
\norm{\nabla_{ij} f(x + d) - \nabla_{ij} f(x)}_2 \leq L_2\norm{d}_2.
\end{equation}
Note that $L_2$ is less than or equal to the Lipschitz constant of the gradient of $f$.

\subsection{Connections between Greedy 2-Coordinate Updates and the 1-Norm}

Our analysis relies on several connections between the greedy update and steepest descent in the 1-norm, which we outline in this section. First, we note that vectors $d^k$ of the form~\eqref{eq:dk} satisfy $\norm{d^k}_1^2 = 2\norm{d^k}_2^2$, since
\begin{align*}
  \norm{d^k}_1^2 &= (|\delta^k| + |-\delta^k|)^2 \\
&= (\delta^k)^2 + (\delta^k)^2 + 2|\delta^k|\cdot|\delta^k| \\
&= 4 (\delta^k)^2 \\
&= 2((\delta^k)^2 + (-\delta^k)^2) \\
&= 2\norm{d^k}_2^2.  
\end{align*}
Second, if a twice-differentiable function's gradient satisfies the 2-coordinate Lipschitz continuity assumption~\eqref{eq:L2-smooth} with constant $L_2$, then the full gradient is Lipschitz continuous in the $1$-norm with constant $L_1=L_2/2$ (see Appendix~\ref{sec:L12}). Finally, we note that applying the 2-coordinate update~\eqref{eq:update2} is an instance of applying steepest descent over all coordinates in the $1$-norm. 
In particular, in Appendix~\ref{proof:sd} we show that steepest descent in the $1$-norm always admits a greedy 2-coordinate update as a solution.
\begin{lemma}\label{lemma:sd}
    Let $\alpha > 0$. Then at least one steepest descent direction with respect to the $1$-norm has exactly two non-zero coordinates.
    That is,
    \[
        \min_{d \in \mathbb{R}^n | d^T 1 = 0} {\nabla f(x)^T d+ \frac{1}{2\alpha} ||d||_1^2}
        = 
        \min_{i,j} \cbr{\min_{d_{ij} \in \mathbb{R}^2 | d_{i} + d_j = 0} {\nabla_{ij} f(x)^T d_{ij}+ \frac{1}{2\alpha} ||d_{ij}||_1^2}}.
        \label{eq:steepest1norm}
    \]
\end{lemma}
This lemma allows us to equate
the progress of greedy 2-coordinate updates to the progress made by a full-coordinate steepest descent step descent step in the 1-norm.

\subsection{Proximal-PL Inequality in the 1-Norm}

For lower bounding sub-optimality in terms of the 1-norm, we introduce the proximal-PL inequality in the 1-norm.
The proximal-PL condition was introduced to allow simpler proofs for various constrained and non-smooth optimization problems~\citep{karimi2016linear}. The proximal-PL condition is normally defined based on the $2$-norm, but we define a variant for the summation-constrained problem where distances are measured in the $1$-norm.
\begin{definition}
A function $f$, that is $L_1$-Lipschitz with respect to the $1$-norm and has a summmation constraint on its parameters, satisfies the proximal-PL condition in the $1$-norm if for a positive constants  $\mu_1$ we have
\begin{equation}
    \frac{1}{2} \mathcal{D}(x,L_1) \geq \mu_1(f(x)-f^*),
    \label{eq:proxL1}
\end{equation}
for all $x$ satisfying the equality constraint. Here, $f^*$ is the constrained optimal function value and
\[
\mathcal{D}(x,L) = -2L \min_{\{y \; | y^T1=\gamma\}} \left[\innerprod{\nabla f(x)}{y-x} +
\frac{L}{2}||y-x||_1^2\right].  \label{eq:Dpl} 
\]
\end{definition}
It follows from the equivalence between norms that summation-constrained functions satisfying the proximal-PL condition in the 2-norm will also satisfy the above proximal-PL condition in the 1-norm. 
In particular, if $\mu_2$ is the proximal-PL constant in the $2$-norm, then we have $\frac{\mu_2}{n} \leq \mu_1 \leq \mu_2$ (see Appendix~\ref{proxpl1norm}).
Functions satisfying these conditions include any strongly-convex function $f$ as well as relaxations of strong convexity, such as functions of the form $f=g(Ax)$ for a strongly-convex $g$ and a matrix $A$~\citep{karimi2016linear}. In the $g(Ax)$ case $f$ is not strongly-convex if $A$ is singular, and we note that the SVM dual problem can be written in the form $g(Ax)$.

\subsection{Convergence Rate of Greedy 2-Coordinate Updates under Proximal-PL}
\label{result}
We analyze the greedy 2-coordinate method under the proximal-PL condition based on the connections to steepest descent in the 1-norm.
\begin{theorem}
\textit{Let $f$ be a twice-differentiable function whose gradient is 2-coordinate-wise Lipschitz~\eqref{eq:L2-smooth} and restricted to the set where $x^T1=\gamma$. If this function satisfies the proximal-PL inequality in the 1-norm~\eqref{eq:proxL1} for some positive $\mu_1$}, then the iterations of the 2-coordinate update~\eqref{eq:update2} with $\alpha^k=1/L_2$ and the greedy rule~\eqref{eq:ik} satisfy:
\begin{equation}
    f(x^k)- f(x^*) \leq \left(1- \frac{2\mu_1}{L_2}\right)^k (f(x^0) - f^*).
    \label{eq:greedyEqnRate}
\end{equation}
\label{eq:mainThm}
\end{theorem}
\begin{proof}
Starting from the descent lemma restricted to the coordinates $i_k$ and $j_k$ we have
\begin{align*}
f(x^{k+1}) & \leq f(x^k) + \nabla_{{i_k}{j_k}} f(x^k)^Td_{{i_k}{j_k}} + \frac{L_2}{2}\norm{d_{{i_k}{j_k}}}^2\\
& = f(x^k) + \min_{i,j}\left\{\min_{\substack{d_{ij}\in\R^2|\\d_i+d_j=0}}\nabla_{ij} f(x^k)^Td_{ij} + \frac{L_2}{2}\norm{d_{ij}}^2\right\} & (\text{GS-q rule})\\
& = f(x^k) + \min_{i,j}\left\{\min_{\substack{d_{ij}\in\R^2 |\\d_i+d_j=0}}\nabla_{ij} f(x^k)^Td_{ij} + \frac{L_2}{4}\norm{d_{ij}}_1^2\right\} & (\text{$\norm{d}_1^2=2\norm{d}^2$})\\
& = f(x^k) + \min_{i,j}\left\{\min_{\substack{d_{ij}\in\R^2 |\\d_i+d_j=0}} \nabla_{ij} f(x^k)^Td_{ij} + \frac{L_1}{2}\norm{d_{ij}}_1^2\right\} & (\text{$L_1 = L_2/2$})\\
& = f(x^k) + \min_{d|d^T1=0}\left\{\nabla f(x^k)^Td + \frac{L_1}{2}\norm{d}_1^2\right\} & (\text{Lemma~\ref{lemma:sd}}).
\end{align*}
Now subtracting $f^*$ from both sides and using the definition of $\mathcal{D}$ from the proximal-PL assumption,
\begin{align*}
f(x^{k+1}) - f(x^*) & \leq f(x^k) - f(x^*)- \frac{1}{2L_1}\mathcal{D}(x^k,L_1)\\
 & = f(x^k) - f(x^*) - \frac{\mu_1}{L_1}(f(x^k) - f^*) \\
& = f(x^k) - f(x^*) - \frac{2\mu_1}{L_2}(f(x^k) - f^*) \\
 &= \left(1- \frac{2\mu_1}{L_2}\right)(f(x^k) - f^*)
\end{align*}
Applying the inequality recursively completes the proof.
\end{proof}
Note that the above rate also holds if we choose $\alpha^k$ to maximally decrease $f$, and the same rate holds up to a constant if we use a backtracking line search to set $\alpha^k$.

\subsection{Comparison to Randomized Selection}

If we sample the two coordinates $i_k$ and $j_k$ from a uniform distribution, then it is known that the 2-coordinate descent method satisfies~\citep{she2017linear}
\begin{equation}
    \E [f(x^k)]- f(x^*) \leq \left(1- \frac{\mu_2}{n^2L_2}\right)^k (f(x^0) - f^*).
    \label{eq:random}
\end{equation}
A similar result for a more-general problem class was shown by~\citet{necoara2014random}. This is substantially slower than the rate we show for the greedy 2-coordinate descent method. This rate is slower even in the extreme case where $\mu_1$ is similar to $\mu_2/n$, due to the presence of the $n^2$ term. There also exist analyses for cyclic selection in the equality-constrained case but existing rates for cyclic rules are slower than the random rates~\citet{wang2014iteration}.

In the case where $f$ is a dense quadratic function of $n$ variables, which includes SVMs under the most popular kernels, both random selection and greedy selection cost $O(n)$ per iteration to implement.
If we consider the time required to reach an accuracy of $\epsilon$ under random selection using the rate~\eqref{eq:random} we obtain $O(n^3\kappa\log(1/\epsilon))$ where $\kappa = L_2/\mu_2$. While for greedy selection under~\eqref{eq:greedyEqnRate} it is between $O(n^2\kappa\log(1/\epsilon))$ if $\mu_1$ is close to $\mu_2/n$ and $O(n\kappa\log(1/\epsilon))$ if $\mu_1$ is close to $\mu_2$. Thus, the reduction in total time complexity from using the greedy method is between a factor of $O(n)$ and $O(n^2)$. This is a large difference which has not been reflected in previous analyses.

There exist faster rates than~\eqref{eq:random} in the literature, but these require additional assumptions such as $f$ being separable or that we know the coordinate-wise Lipschitz constants~\citep{necoara2011random,necoara2014random,necoara2017random,fang2018faster}. However, these assumptions restrict the applicability of the results. Further, unlike convergence rates for random coordinate selection, we note that the new linear convergence rate~\eqref{eq:greedyEqnRate} for greedy 2-coordinate method avoids requiring a direct dependence on the problem dimension. 
The only previous dimension-independent convergence rate for the greedy 2-coordinate method that we are aware of is due to~\citet[][Theorem 5.2b]{beck20142}.
Their work considers functions that are bounded below, which is a weaker assumption than the proximal-PL assumption. However, this only leads to sublinear convergence rates and only on a measure of the violation in the  Karush-Kuhn-Tucker conditions.~\citet[][Theorem 6.2]{beck20142} also gives convergence rates in terms of function values for the special case of convex functions, but these rates are sublinear and dimension dependent.

\section{Equality- and Bound-Constrained Greedy Coordinate Updates}

Equality constraints often appear alongside lower and/or upper bounds on the values of the individual variables. This results in problems of the form
\begin{equation}
\min_{x\in\R^n} f(x), \hspace{0.3 cm}\text{subject to     } \sum\limits_{i=1}^n x_i =\gamma, \; l_i \leq x_i \leq u_i.
\label{eq:bounds}
\end{equation}
This framework includes our motivating problems of optimizing over the probability simplex ($l_i=0$ for all $i$ since probabilites are non-negative), and optimizing SVMs with an unregularized bias (where we have lower and upper bounds). With bound constraints we use a $d^k$ of form~\eqref{eq:dk} but where $\delta^k$ is defined so that the step respects the constraints,
\begin{equation}
\begin{aligned}
\delta^k = -\min\Big\{& \frac{\alpha^k}{2}(\nabla_{i_k}f(x^k) - \nabla_{j_k}f(x^k)),
 x_{i_k}^k - l_{i_k}, u_{j_k} - x_{j_k}^k\Big\},
\end{aligned}
\label{eq:deltaBound}
\end{equation}
Unfortunately, analyzing the bound-constrained case is more complicated. There are several possible generalizations of the greedy rule for choosing the coordinates $i_k$ and $j_k$ to update, depending on what properties of~\eqref{eq:ik} we want to preserve~\citep[see][Section~2.7]{nutini2018greed}. In this section we discuss several possibilities, and how the choice of greedy rule affects the convergence rate and iteration cost.

\subsection{GS-s Rule: Minimizing Directional Derivative}

Up until version 2.7, the greedy rule used in LIBSVM was the Gauss-Southwell-s (GS-s) rule. The GS-s rule chooses the coordinates resulting in the $d^k$ with the most-negative directional derivative. This is a natural generalization of the idea of steepest descent, and the first uses of the method that we aware of are by~\citet{keerthi2001improvements} for SVMs and by~\citet{shevade2003simple} for 1-norm regularized optimization. For problem~\eqref{eq:bounds} the GS-s rule chooses
\[
i_k \in \argmax_{i \;|\; x_i^k > l_i}\nabla_i f(x^k), \quad j_k \in \argmin_{j \; | \; x_j^k < u_i}\nabla_j f(x^k).
\]
This is similar to the unbounded greedy rule~\eqref{eq:ik} but excludes variables where the update would immediately violate a bound constraint. 

Unfortunately, the per-iteration decrease in $f$ obtained by the GS-s rule can be arbitrarily small. In particular, consider the case where the variable $i$ maximizing $\nabla_i f(x^k)$ has a value of $x_i^k=l_i+\epsilon$ for an arbitrarily small $\epsilon$. In this case, we would choose $i_k$ and take an arbitrarily small step of $\delta^k=\epsilon$. Steps like this that truncate $\delta^k$ are called ``bad'' steps, and the GS-s rule does not guarantee a non-trivial decrease in $f$ on bad steps. If we only have bound constraints and do not have an equality constraint (so we can update on variable at a time),~\citet{karimireddy2019efficient} show that at most half of the steps are bad steps. Their argument is that after we have taken a bad step on coordinate $i$, then the next time $i$ is chosen we will not take a bad step. However, with an equality constraint it is possible for a coordinate to be involved in consecutive bad steps. It is possible that a combinatorial argument similar to~\citet[][Theorem~8]{lacoste2015global} could bound the number of bad steps, but it is not obvious that we do not require an exponential total number of bad steps.

\subsection{GS-q Rule: Minimum 2-Coordinate Approximation}

A variant of the Gauss-Southwell-q (GS-q) rule of~\citet{tseng2009block}for problem~\eqref{eq:bounds} is
\begin{equation}
\begin{aligned}
\argmin_{i,j}\min_{d_{ij}|d_i+d_j=0} \left\{f(x^k) 
+ \nabla_{ij} f(x^k)^Td_{ij} + \frac{1}{2\alpha^k}\norm{d_{ij}}^2 : 
x^k + d \in [l, u]
\right\}.
\end{aligned}
\label{eq:GS-q}
\end{equation}
This minimizes a quadratic approximation to the function, restricted to the feasible set. For problem~\eqref{eq:bounds}, the GS-q rule is equivalent to choosing $i_k$ and $j_k$ to maximize~\eqref{eq:deltaBound}, the distance that we move. We show the following result for the GS-q rule in Appendix~\ref{app:GSq}.
\begin{theorem}
\textit{Let $f$ be a differentiable function whose gradient is 2-coordinate-wise Lipschitz~\eqref{eq:L2-smooth} and restricted to the set where $x^T1=\gamma$ and $l_i \leq x_i \leq u_i$. If this function satisfies the proximal-PL inequality in the 2-norm~\citep{karimi2016linear} for some positive $\mu_2$}, then the iterations of the 2-coordinate update~\eqref{eq:dk} with $\delta^k$ given by~\eqref{eq:deltaBound}, $\alpha^k=1/L_2$, and the greedy GS-q rule~\eqref{eq:GS-q} satisfy:
\[
    f(x^k)- f(x^*) \leq \left(1- \frac{\mu_2}{L_2(n-1)}\right)^k (f(x^0) - f^*).
\]
\end{theorem}
The proof of this result is more complicated than our previous results, relying on the concept of conformal realizations used by~\citet{tseng2009block}. 
We prove the result for general block sizes and then specialize
to the two-coordinate case.
Unlike the GS-s rule, this result shows that the GS-q guarantees non-trivial progress on each iteration. Note that while this result does have a dependence on the dimension $n$, it does not depend on $n^2$ as the random rate~\eqref{eq:random} does.
Moreover, the dependence on $n$ can be improved by increasing the block size.

Unfortunately, the GS-q rule is not always efficient to use. As discussed by~\citet{beck20142}, there is no known algorithm faster than $O(n^2)$ for computing the GS-q rule~\eqref{eq:GS-q}. One special case where this can be solved in $O(n)$ given the gradient is if we only have lower bounds (or only have upper bounds)~\citep{beck20142}. An example with only lower bounds is our motivating problem of optimizing over the probability simplex, which only requires variables to be non-negative and sum to 1. On the other hand, our other motivating problem of SVMs requires lower and upper bounds so computing the GS-q rule would require $O(n^2)$. Beginning with version 2.8, LIBSVM began using an approximation to the GS-q rule that can be computed in $O(n)$. In particular, LIBSVM first chooses one coordinate using the GS-s rule, and then optimizes the other coordinate according to a variant of the GS-q rule~\citep{fan2005working}.\footnote{The newer LIBSVM rule also uses Lipschitz information about each coordinate; see Section~\ref{sec:Li} for discussion.} While other rules have been proposed, the LIBSVM rule remains among the best-performing methods in practice~\citep{horn2018comparative}. However, similar to the GS-s rule we cannot guarantee non-trivial progress for the practical variant of the GS-q rule used by LIBSVM.

\subsection{GS-1 Rule: Steepest Descent in the 1-Norm}

\begin{algorithm}
\caption{The GS-1 algorithm (with variables sorted in descending order according to $\nabla f(x)$).}\label{alg:euclid}
\begin{algorithmic}[1]
\Function{GS-1}{$x,\nabla f(x),\alpha,l,u$}
\State $x_{0} \gets 0; x_{n+1} \gets 0; i \gets 1; j \gets n; d \gets 0;$
\While 1
\State $\delta \gets \frac{\alpha}{4} \left(\nabla_i f(x) - \nabla_j f(x) \right)$
\State $\omega = \sum\limits_{p=0}^{i-1} x_p-l_p$; $\kappa  = \sum\limits_{q=j+1}^{n+1} u-x_q$
\If{$\delta- \omega  < 0\And 
\delta - \kappa < 0$}
    \If {$\omega < \kappa$} $d_i = \omega  - \kappa$ ; break;
    \Else \hspace{0.1cm}$d_j =\omega  - \kappa$; break;
    \EndIf
\ElsIf {$\delta- \omega  < 0$} $d_j = \omega  - \kappa$; break;
\ElsIf {$\delta- \kappa < 0$} $d_i = \omega  - \kappa$; break;
\EndIf
\If{$x_i +\omega -\delta \geq l_i \And x_j - \kappa+\delta \leq u_j$}
    \State $d_i = \omega -\delta ; d_j = \delta-\kappa$; break;
\EndIf
\If{$x_i +\omega -\delta < l_i \And x_j - \kappa+\delta > u_j$}
    \If{$l_i - (x_i +\omega -\delta) > x_j - \kappa+\delta- u_j$}
    \State $d_i = l- x_i$; $i\gets i+1$
    \Else
    \State $d_j = u- x_j$; $j\gets j-1$
    \EndIf
\ElsIf{$x_i+\omega  - \delta< l_i$} $d_i=l-x_i$; $i\gets i+1$
\Else \hspace{0.1cm}$d_j = u- x_j$; $j\gets j-1$
\EndIf
\EndWhile
\State \textbf{return} $d$
\EndFunction
\end{algorithmic}
\end{algorithm}

Rather than using the classic GS-s or GS-q selection rules, the Gauss-Southwell-1 (GS-1) rule performs steepest descent in the 1-norm. For problem~\eqref{eq:bounds} this gives the update
\begin{equation}
d^k \in \argmin_{l_i \leq x_i+d_i \leq u_i | d^T 1 = 0} \left\{\nabla f(x^k)^T d+ \frac{1}{2\alpha^k} ||d||_1^2\right\}.
\label{eq:GS-1}
\end{equation}
The GS-1 rule was proposed by~\citet{song2017accelerated} for (unconstrained) 1-norm regularized problems. To analyze this method, we modify the definition of $\mathcal{D}(x,L)$ in the proximal-PL assumption to be
\begin{align}
\mathcal{D}(x,L) = & -2L \min_{\{l_i \leq y_i \leq u_i \; | y^T1=\gamma\}} \Big\{\innerprod{\nabla f(x)}{y-x} +
\frac{L}{2}||y-x||_1^2\Big\}.  
\label{eq:Dbound}
\end{align}
We then have the following dimension-independent convergence rate for the GS-1 rule.
\begin{theorem}
\textit{Let $f$ be a differentiable function whose gradient is 2-coordinate-wise Lipschitz~\eqref{eq:L2-smooth} and restricted to the set where $x^T1=\gamma$ and $l_i \leq x_i \leq u_i$. If this function satisfies the proximal-PL inequality in the 1-norm~\eqref{eq:proxL1} for some positive $\mu_1$} with the definition~\eqref{eq:Dbound}, then the iterations of the update $x^{k+1} = x^k + d^k$ with the greedy rule~\eqref{eq:GS-1} and $\alpha_k=1/L_1 = 2/L_2$ satisfy:
\[
    f(x^k)- f(x^*) \leq \left(1- \frac{2\mu_1}{L_2}\right)^k (f(x^0) - f^*).
    \label{eq:greedyRate}
\]
\end{theorem}
\begin{proof}
The proof follows the same reasoning as Theorem~\ref{eq:mainThm}, but beginning after the application of Lemma~\ref{lemma:sd} since we are directly computing the steepest descent direction.
\end{proof}
This GS-1 convergence rate is at least as fast as the convergence rate for GS-q, and thus by exploiting a connection to the 1-norm we once again obtain a faster dimension-independent rate. 
In Algorithm~\ref{alg:euclid} we give a method to construct a solution to the GS-1 rule~\eqref{eq:GS-1} in $O(n \log n)$ time (due to sorting the $\nabla_i f(x^k)$ values). Thus, our new GS-1 update guarantees non-trivial progress at each step (unlike the GS-s rule) and is efficient to compute (unlike the GS-q rule). The precise logic of Algorithm~\ref{alg:euclid} is somewhat complicated, but it can intuitively be viewed as a version of GS-s that fixes the bad steps where $\delta^k$ is truncated. Roughly, if the GS-s rule gives a bad step then the GS-1 moves the violating variable to its boundary and then may also update the variable with the next largest/smallest $\nabla_if(x^k)$.

The drawback of the GS-1 update is that it is not strictly a 2-coordinate method. While the GS-1 update moves at most 2 variables within the interior of the bound constraints, it may move additional variables to their boundary. The iteration cost of the method will be higher on iterations where more than 2 variables are updated. However, by using an argument similar to~\citet{sun2019we}, we can show that the GS-1 rule will only update more than 2 variables on a finite number of early iterations. This is because, after some finite number of iterations, the variables actively constrained by their bounds will remain at their bounds. At this point, each GS-1 update will only update 2 variables within  the interior of the bounds. In the case of SVMs, moving a variable to its lower bound corresponds to removing it as a potential support vector. Thus, this ``bug'' of GS-1 that it may update more than 2 variables can allow it to quickly remove many support vectors. In our experiments, we found that GS-1 identified the support vectors more quickly than other rules and that most GS-1 updates only updated 2 or 3 coordinates. 

\section{Greedy Updates using Coordinate-Wise Lipschitz Constants}
\label{sec:Li}

Up until this point, we have measured smoothness based on the maximum blockwise Lipschitz-constant $L_2$. An alternative measure of smoothness is Lipschitz continuity of individual coordinates. In particular, coordinate-wise Lipschitzness of coordinate $i$ requires that for all $x$ and $\alpha$ 
\[
|\nabla_i f(x + \alpha e_i) - \nabla_i f(x)| \leq L_i|\alpha|,
\]
where $e_i$ is a vector with a one in position $i$ and zeros in all other positions. For twice-differentiable convex functions, the Lipschitz constant with respect to the block $(i,j)$ is upper bounded by the sum of the coordinate-wise constants $L_i$ and $L_j$~\citep[Lemma~1]{nesterov2012efficiency}. For equality-constrained problems, this leads to a coordinate descent update of the form
\begin{equation}
\delta^k = -(\nabla_{i_k} f(x^k) - \nabla_{j_k} f(x^k))/(L_{i_k}+L_{j_k}).
\label{eq:deltaLi}
\end{equation}
\citet{necoara2011random} uses the coordinate-wise Lipschitz constants to design sampling distributions for $i_k$ and $j_k$ for equality-constrained optimization. Their analysis gives rates that can be faster than uniform sampling~\eqref{eq:random}. 

In Appendix~\ref{app:Li}, we show that the equality-constrained GS-q rule incorporating the $L_i$ values into a weighted 2-norm is given by
\begin{equation}
\argmax_{i,j}\left\{\frac{\nabla_i f(x^k) - \nabla_j f(x^k)}{\sqrt{L_i + L_j}}\right\}.
\label{eq:GSL}
\end{equation}
On the other hand, the GS-1 rule under a weighted 1-norm is given by
\begin{equation}
\argmax_{i,j}\left\{\frac{\nabla_i f(x^k) - \nabla_j f(x^k)}{\sqrt{L_i} + \sqrt{L_j}}\right\}.
\label{eq:GSL-1}
\end{equation}
Thus, the steepest descent equivalence does hold even without bound constraints. However, both~\eqref{eq:GSL} and~\eqref{eq:GSL-1} yield the standard greedy rule~\eqref{eq:ik} if all $L_i$ values are equal.

Unfortunately, it is not obvious how to solve~\eqref{eq:GSL} or~\eqref{eq:GSL-1} faster than $O(n^2)$. 
In our experiments we explored an approximation that can be computed in $O(n)$ given the gradient,
\begin{equation}
i_k \in \argmax_i (\nabla_i f(x^k)- \mu)/\sqrt{L_i}, \quad j_k \in \argmin_j(\nabla_j f(x^k)-\mu)/\sqrt{L_j},
\label{eq:ratio}
\end{equation}
where $\mu$ is the mean of $\nabla f(x^k)$. We call~\eqref{eq:ratio} the ratio approximation. This approximation incorporates the $\sqrt{L_i}$ values into the greedy rule, and by subtracting the mean we are guaranteed to select an $i_k$ where $\nabla_{i_k}f(x)$ is above the mean a $j_k$ where $\nabla_{j_k}f(x)$ is below the mean.
Further, we can implement efficient variations of the GS-s and GS-1 rules based on this approximation if we also have bound constraints.





%
%
\section{Experiments}

Our first experiment evaluates the performance of various rules on a synthetic equality-constrained least squares problem. Specifically, the objective is $f(x) = \frac 1 2 ||Ax-b||^2$ subject to $x^T 1 = 0$. We generate the elements of $A \in \R^{1000 \times 1000}$ from a standard normal and set $b=Ax + z$ where $x$ and $z$ are generated from standard normal distributions. We also consider a variant where each column of $A$ is scaled by a sample from a standard normal to induce very-different $L_i$ values. In Figure~\ref{fig:randomvsgreedy} we compare several selection rules: random $i_k$ and $j_k$, the greedy rule~\eqref{eq:ik}, sampling $i_k$ and $j_k$ proportional to $L_i$, the exact greedy $L_i$ rule~\eqref{eq:GSL}, the ratio greedy $L_i$ rule, and GS-1 $L_i$ rule ~\eqref{eq:GSL-1}. All algorithms use the update~\eqref{eq:deltaLi}. In these experiments we see that greedy rules lead to faster convergence than random rules in all cases. We see that knowing the $L_i$ values does not significantly change the performance of the random method, nor does it change the performance of the greedy methods in the case when the $L_i$ were similar. However, with different $L_i$ the greedy methods exploiting the $L_i$ values work better than the greedy method that does not use the $L_i$. This includes the ratio approximation~\eqref{eq:ratio} which does not have the high cost of the greedy GS-q and GS-1 rules that exploit the $L_i$ values. 

\begin{figure}[ht!]
    \centering
    \includegraphics[width=0.49\textwidth]{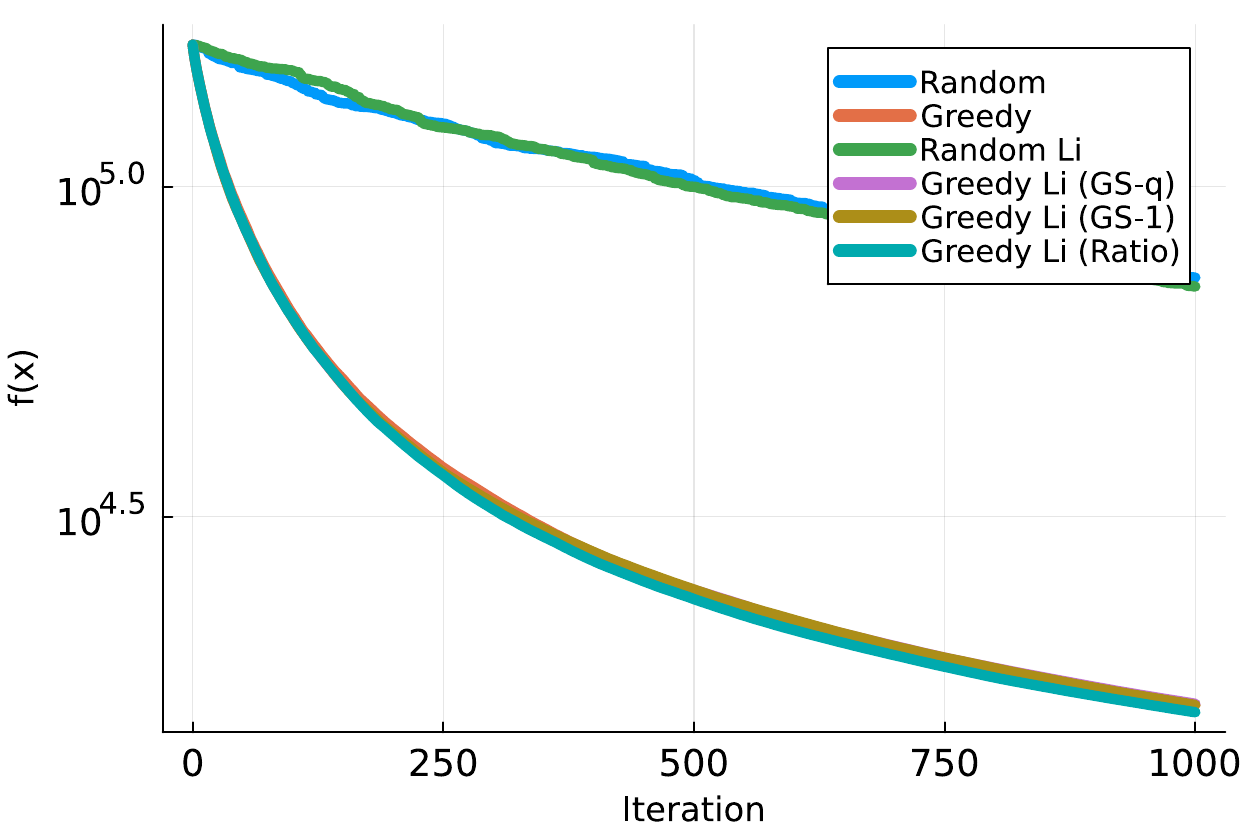}
    \includegraphics[width=0.49\textwidth]{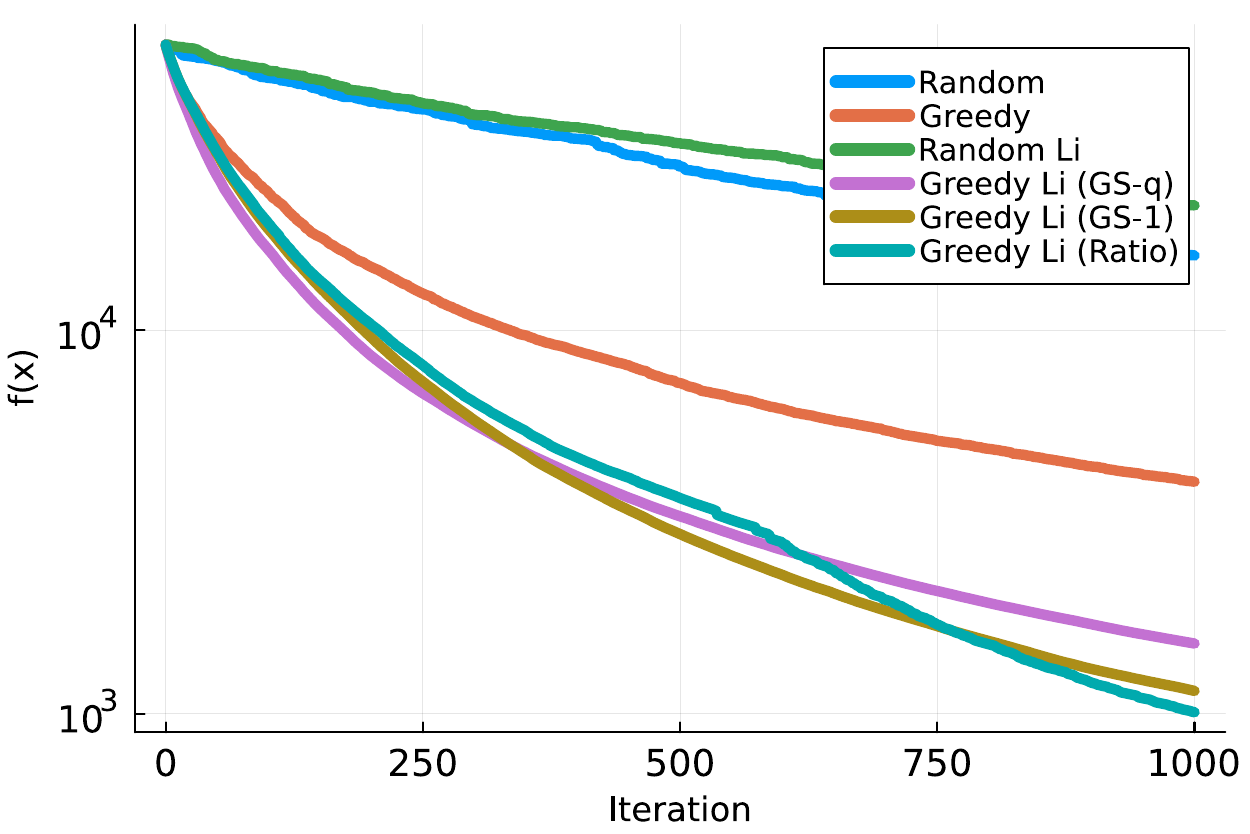}
    \caption{Random vs greedy coordinate selection rules, including rules using the coordinate-wise Lipschitz constants $L_i$. The $L_i$ are similar in the left plot, but differ significantly on the right.}
    \label{fig:randomvsgreedy}
\end{figure}

Our second experiment considers the same problem but with the additional constraints $x_i \in [-1,1]$. Figure~\ref{fig:boundexp} compares the GS-s, GS-q, and GS-1 rules in this setting. We see that the GS-s rule results in the slowest convergence rate, while the GS-q rule rule takes the longest to identify the active set. The GS-1 rule typically updates 2 variables, but occasionally updates 3 variables and rarely updates 4 or 5 variables.

\begin{figure}[H]
    \centering    
    \includegraphics[width=0.32\textwidth]{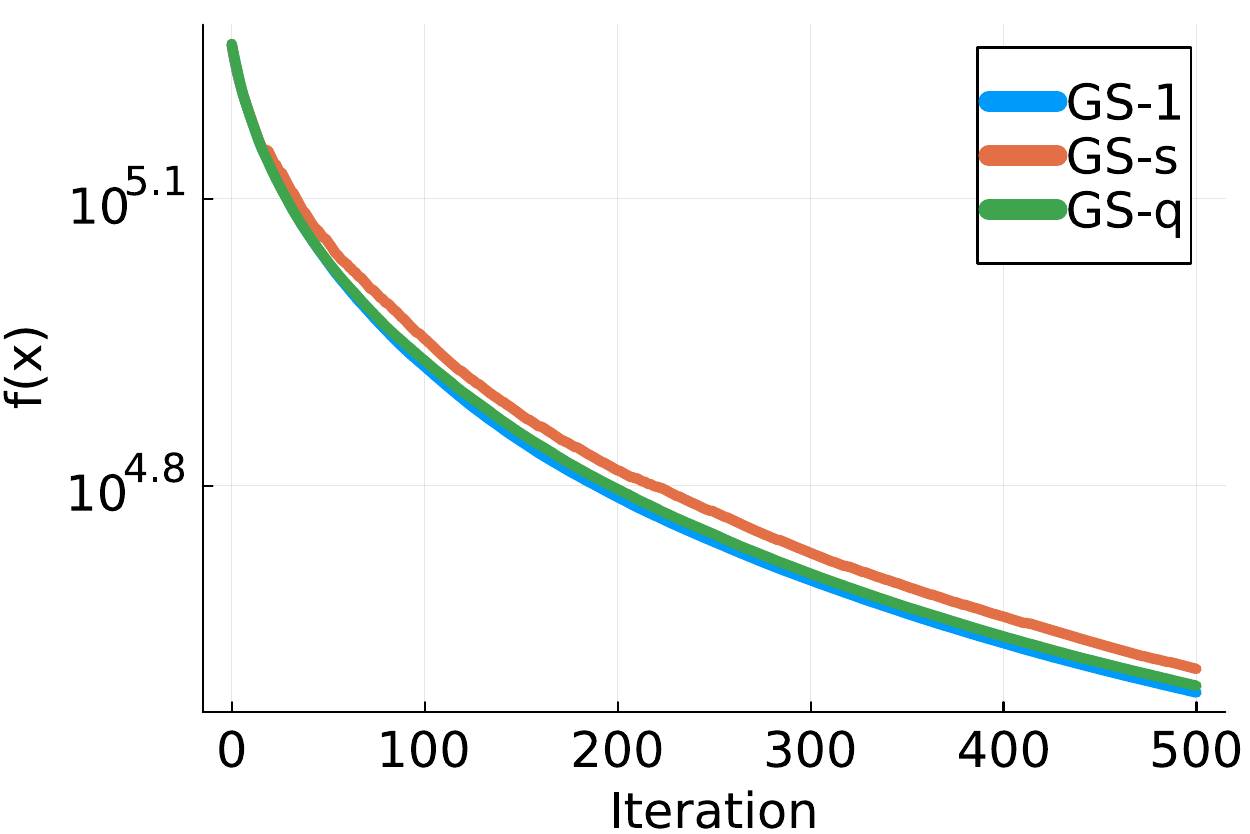}
    \includegraphics[width=0.32\textwidth]{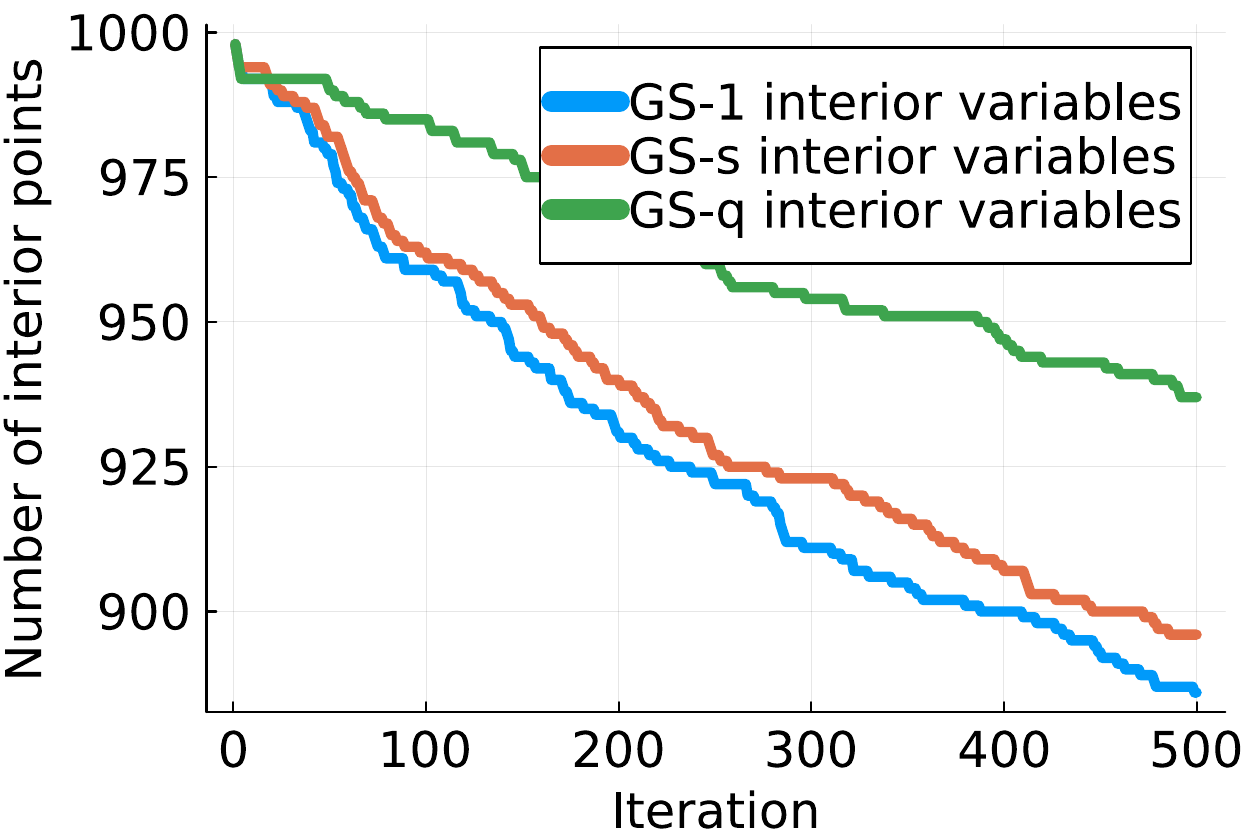}
    \includegraphics[width=0.32\textwidth]{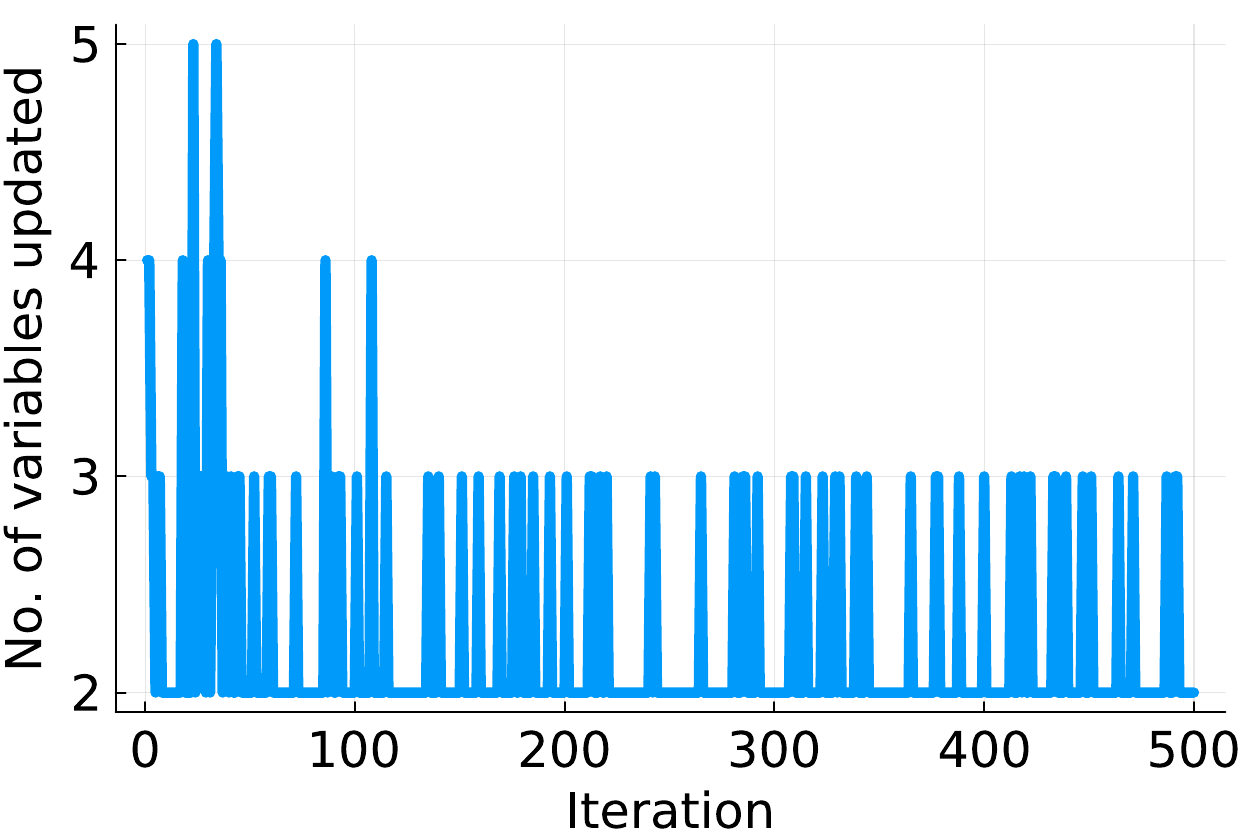}
    \caption{Comparison of GS-1, GS-q and GS-s under linear equality constraint and bound constraints. The left plot shows the function values, the middle plot shows the number of interior variables, and the right plot shows the number of variables updated by the GS-1 rule.}
    \label{fig:boundexp}
\end{figure}
\section{Discussion}
Despite the popularity of LIBSVM, up until this work we did not have a strong justification for using greedy 2-coordinate methods over simpler random 2-coordinate methods for equality-constrained optimization methods. This work shows that greedy methods may be faster by a factor ranging from $O(n)$ up to $O(n^2)$. This work is the first to identify the equivalence between the greedy 2-coordinate update and steepest descent in the 1-norm. The connection to the 1-norm is key to our simple analyses and also allows us to analyze greedy rules depending on coordinate-wise Lipschitz constants.

For problems with bound constraints and equality constraints, we analyzed the classic GS-q rule but also proposed the new GS-1 rule. Unlike the GS-s rule the GS-1 rule guarantees non-trivial progress on each iteration, and unlike the GS-q rule the GS-1 rule can be implemented in $O(n \log n)$. We further expect that the GS-1 rule could be implemented in $O(n)$ by using randomized algorithms, similar to the techniques used to implement $O(n)$-time projection onto the 1-norm ball~\citep{duchi2008efficient,vangroup}. The disadvantage of the GS-1 rule is that on some iterations it may update more than 2 coordinates on each step. However, when this happens the additional coordinates are simply moved to their bound. This can allow us to identify the active set of constraints more quickly. For SVMs this means identifying the support vectors faster, giving cheaper iterations.

\section*{Acknowledgements}
This research was partially supported by the
Canada CIFAR AI Chair Program, the Natural Sciences and Engineering Research Council of Canada
(NSERC) Discovery Grants RGPIN-2022-03669.

\bibliography{example_paper}
\bibliographystyle{plainnat}

\newpage
\appendix
\section{Equivalent Ways of Writing Equality-Constrained Greedy Rule}
\label{def}

We first show that the greedy rule with a summation constraint, of choosing the max/min partial derivatives, is an instance of the GS-q rule. We then show that this rule is also equivalent to steepest descent in the 1-norm.

\subsection{Greedy Rule Maximizes GS-q Progress Bound}\label{proof:b_and_d}

For the optimization problem \eqref{eq:BCD}, the GS-q rule selects the optimal block $b=\{i,j\}$, by solving the following minimization problem:
\begin{equation}
	\label{eq:gsq rule}
	b = \argmin_{b} \left\{\min_{d_b|d_i +d_j = 0} \langle \nabla_b f(x), d_b \rangle + \dfrac{1}{2\alpha}||d_b||^2\right\},
\end{equation}
where $d_b$ is the descent direction.
\paragraph{Solving for $d_b$.}
First let us fix $b$ and solve for $d_b$.
The Lagrangian of \eqref{eq:gsq rule} is,
\begin{equation*}
	\mathcal{L}(d_b, \lambda) =  \langle \nabla_b f(x), d_b \rangle + \dfrac{1}{2\alpha}||d_b||^2 + \lambda(d_1 + d_2).
\end{equation*}
Taking the gradient with respect to $d_b$ gives,
\begin{equation*}
	\nabla_{d_b} \mathcal{L}(d_b,\lambda) = \nabla_b f(x) + \frac{1}{\alpha} d_b + \lambda 1.
\end{equation*}
Setting the gradient equal to $0$ and solving for $d_b$ gives,
\begin{align}
	d_b = -\alpha (\nabla_b f(x) + \lambda 1). \label{eq:db}
\end{align}
From our constraint, $d_i+ d_j= 0$, we get
\begin{align*}
	0       & = -\alpha \left (\nabla_i f(x) + \lambda + \nabla_j f(x) + \lambda \right), \\
	\lambda & = -\frac 1 2 \langle\nabla_b f(x),1 \rangle.
\end{align*}
Substituting in \eqref{eq:db} we get,
\begin{equation}
	d_b = - \alpha \left ( \nabla_b f(x) - \frac 1 2 \langle \nabla_b f(x),1\rangle1 \right).
	\label{eq:optdb}
\end{equation}
This can be re-written as
\begin{align*}
	\begin{bmatrix}
		d_i \\
		d_j \\
	\end{bmatrix}
	=
	\frac{\alpha}{2} \left(\nabla_i f(x) - \nabla_j f(x)\right)
	\begin{bmatrix}
		-1 \\
		1  \\
	\end{bmatrix}.
\end{align*}

\paragraph{Solving for $b$.}
Now, we plug in the optimal $d_b$ from \eqref{eq:optdb} in   \eqref{eq:gsq rule} and solve for $b$ to give
\begin{equation}
\begin{aligned}
	 & \argmin_{b} - \alpha \left\langle \nabla_b f(x),  ( \nabla_b f(x) - \frac 1 2 \langle \nabla_b f(x),1\rangle1) \right\rangle + \frac {\alpha} {2}||( \nabla_b f(x) - \frac 1 2 \langle \nabla_b f(x),1\rangle1)||^2 \\
	 & \equiv \argmin_{b} - ||\nabla_b f(x)||^2 + \frac{1}{2}(\langle \nabla_b f(x),1\rangle)^2 + \frac 1 2 ||\nabla_b f(x)||^2 - \frac 1 2 (\langle \nabla_b f(x),1\rangle)^2 +                                         \\
	 & \hspace{1.5 cm} \frac 1 8 (\langle \nabla_b f(x),1\rangle)^2 \underbrace{\langle1,1\rangle}_2                                                                \\
	 & \equiv \argmin_{b} -\frac 1 2 ||\nabla_b f(x)||^2 + \frac 1 4 (\langle \nabla_b f(x),1\rangle)^2                                                                                                                \\
	 & \equiv \argmax_{b} ||\nabla_b f(x)||^2 - \frac 1 2 (\langle \nabla_b f(x),1\rangle)^2                                                                                                                           \\
	 & \equiv \argmax_{b}||\nabla_b f(x)||^2 - \frac 1 2 (\nabla_i f(x)+ \nabla_j f(x))^2                                                                                                                                 \\
	 & \equiv \argmax_{b} \frac 1 2 ||\nabla_b f(x)||^2 - \nabla_if(x)\nabla_jf(x)                                                                                                                                         \\
	 & \equiv \argmax_{b} \frac 1 2 (\nabla_i f(x) - \nabla_j f(x))^2)                                                                                                                                                      \\
	 & \equiv \argmax_{b} |\nabla_i f(x)- \nabla_j f(x)|. \label{eq:solve4b}
\end{aligned}
\end{equation}
Therefore, the GS-q rule chooses the $i$ and $j$ that are farthest apart, which are the coordinates with maximum and minimum values in $\nabla f(x)$.


\subsection{Greedy Rule is Steepest Descent in the 1-Norm (Lemma \ref{lemma:sd})}
\label{proof:sd}

The steepest descent method finds the descent direction that minimizes the function value in every iteration. That is,
\begin{equation}
	d = \argmin_{d \in \mathbb{R}^n | d^T 1 = 0} \left\{\nabla f(x)^T d+ \frac{1}{2\alpha}\|d\|_1^2\right\}. \label{eq:SD}
\end{equation}
The proof follows by constructing a solution to the steepest descent
problem~\eqref{eq:SD} which only has two non-zero entries.
The Lagrangian of \eqref{eq:SD} is,
\begin{equation*}
	\mathcal{L}(d,\lambda) = \nabla f(x)^T d+ \frac{1}{2\alpha} \|d\|_1^2 + \lambda d^T 1.
\end{equation*}
The sub-differential with respect to $d$ and $\lambda$ is given by
\begin{align*}
	\partial_d \mathcal{L}(d,\lambda)         & \equiv \nabla f(x) + \frac{1}{2\alpha} \partial \|d\|_1^2 + \lambda 1, \\
	\partial_{\lambda} \mathcal{L}(d,\lambda) & \equiv d^T 1.
\end{align*}
We have that the zero vector is an element of the sub-differential at the solution. From $0\in\partial_\lambda\mathcal{L}(d,\lambda)$ we have $d^T 1 =0$. From $0 \in \partial_d \mathcal{L}(d,\lambda)$ at the solution we require
\[
	2\alpha (-\nabla f(x)-\lambda 1) \in \partial \|d\|_1^2,
\]
or equivalently by using that $\partial_i\norm{d}_1^2 \equiv 2\norm{d}_1\text{sgn}(d_i)$ this subgradient inclusion is equivalent to having for each coordinate $i$ that
\begin{equation}
	\alpha (-\nabla_i f(x)-\lambda 1) = \|d\|_1 \text{sgn}(d_i),\label{eq:SDpartial}
\end{equation}
where the signum function sgn$(d_i)$ is $+1$ if $d_i$ is positive, $-1$ if $d_i$ is negative, and can take any value in the  interval $[-1,1]$ if $d_i$ is zero.

Let $i \in \argmax_i \{\nabla_i f(x)\}$ and $j \in \argmin_j \{\nabla_j f(x)\}$. Consider a solution $d$ such that $d_i = \delta, d_j = -\delta$ for some $\delta \in \R$ and $d_k=0$ for if $k\neq i$ and $k\neq j$.
By construction the vector $d$ has only two non-zero coordinates and satisfies the sum-to-zero constraint required for feasibility. Thus, we have a solution if we can choose $\delta$ to satisfy~\eqref{eq:SDpartial} for all coordinates.

The definition of $d$ implies $||d||_1= 2\delta$,
while sgn$(d_i)=1$, sgn$(d_j)=-1$ and sgn$(d_k) \in [-1,1]$.
Thus, for $d$ to be a steepest descent direction we must have:
\begin{align}
	-\alpha \nabla_i f(x) -\alpha\lambda  & = 2\delta \label{eq:cond1}          \\
	-\alpha \nabla_j f(x) - \alpha\lambda & =  -2\delta\label{eq:cond2}         \\
	-\alpha \nabla_k f(x) - \alpha\lambda & \in 2\delta[-1,1]. \label{eq:cond3}
\end{align}
Solving for $\lambda$ in \eqref{eq:cond1} gives
\begin{equation}
	\lambda = -\nabla_i f(x) - 2\delta/\alpha, \label{eq:lambda}
\end{equation}
and substituting this in \eqref{eq:cond2} gives,
%
\begin{equation}
	\delta = -\frac{\alpha}{4}(\nabla_i f(x) - \nabla_j f(x)).\label{eq:delta}
\end{equation}
It remains only to show that
\eqref{eq:cond3} is satisfied by $d$.
Using the value of $\lambda$~\eqref{eq:lambda} in~\eqref{eq:cond3} yields,
\begin{align*}
	-\alpha \nabla_k f(x) + \alpha \nabla_i f(x) +  2\delta & \in 2\delta[-1,1].
\end{align*}
Now, substituting the value for $\delta$~\eqref{eq:delta} gives
\[
	-\alpha\nabla_k f(x) + \alpha\nabla_i f(x) - \frac{\alpha}{2}(\nabla_i f(x) - \nabla_j f(x)) \in -\frac{\alpha}{2}(\nabla_i f(x) - \nabla_j f(x))[-1,1],
\]
and multiplying by $2/\alpha$ this is equivalent to
\[
	-2\nabla_k f(x)+\nabla_i f(x) + \nabla_j f(x) \in -(\nabla_i f(x) - \nabla_j f(x))[-1,1],\]
which can be satisfied for some value in $[-1,1]$ if
\[
	-2\nabla_k f(x)+\nabla_i f(x) + \nabla_j f(x) \leq |\nabla_i f(x) - \nabla_j f(x)|.
\]
As $\nabla_k f(x)$ is between $\nabla_i f(x)$ and $\nabla_j f(x)$, we can write it as a convex combination $\theta \nabla_i f(x) + (1-\theta) \nabla_j f(x)$ for some $\theta\in[0,1]$. Thus, we require
\begin{align*}
	 & -2(\theta \nabla_i f(x) + (1-\theta) \nabla_j f(x)) +
	\nabla_i f(x) + \nabla_j f(x)                            \\
	 & =
	(1-2 \theta) (\nabla_i f(x) - \nabla_j f(x)) \leq |\nabla_i f(x) - \nabla_j f(x)|,
\end{align*}
which holds because $(1-2\theta) \in [-1,1]$.

We have shown that a two-coordinate update $d$ satisfies the sufficient conditions to be a steepest descent direction in the $1$-norm.
Substituting $d$ back into the expression for steepest descent gives
\begin{align*}
	\min_{d \in \mathbb{R}^n | d^T 1 = 0} {\nabla f(x)^T d+ \frac{1}{2\alpha} ||d||_1^2}
	 & = \nabla_{ij} f(x)^T d_{ij}+ \frac{1}{2\alpha} ||d_{ij}||_1^2                                                                          \\
	 & \geq \min_{i,j} \cbr{\min_{d_{i,j} \in \mathbb{R}^2 | d_{i} + d_j = 0} {\nabla_{ij} f(x)^T d_{ij}+ \frac{1}{2\alpha} ||d_{ij}||_1^2}}.
\end{align*}
The reverse inequality follows from the fact that a two coordinate update cannot
lead to a smaller value than updating all coordinates, so we have
\begin{align*}
	\min_{d \in \mathbb{R}^n | d^T 1 = 0} {\nabla f(x)^T d+ \frac{1}{2\alpha} ||d||_1^2}
	 & = \min_{i,j} \cbr{\min_{d_{i,j} \in \mathbb{R}^2 | d_{i} + d_j = 0} {\nabla_{ij} f(x)^T d_{ij}+ \frac{1}{2\alpha} ||d_{ij}||_1^2}}.
\end{align*}


\section{Relating Lipschitz Constants}
\label{sec:L12}

\begin{proposition}
	Suppose $f$ is twice differentiable and
	\begin{equation}\label{eq:ell1-smoothness}
		\sup_{x : \abr{x, 1} = a } \max_{d} \cbr{ d^\top \nabla^2 f(x)d : \abr{d, 1} = 0, \text{supp}(d) = 2 , \|d\|_1 \leq 1 } = L_1.
	\end{equation}
	Then $f$ satisfies the following inequality:
	\begin{equation*}
		f(x + d) \leq f(x) + \abr{ \nabla f(x), d} + \frac{L_1}{2}\norm{d}_1^2,
	\end{equation*}
	for $x$ such that \( \abr{x, 1} = a \) and any $d$ such that
	$\abr{d, 1} = 0$.
	That is, \( f \) is full-coordinate Lipschitz smooth in the \( \ell_1 \)
	norm with constant \( L_1 \).
\end{proposition}
\begin{proof}
	Consider the optimization problem
	\begin{equation}\label{eq:variational-smoothness}
		\max_{d} \cbr{ d^\top \nabla^2 f(x)d : \abr{d, 1} = 0, \norm{d}_1 \leq 1 }.
	\end{equation}
	We will show that the maximum is achieved by at least one \( d \)
	satisfying \( d_i = - d_j \neq 0 \), \( d_k = 0 \) for all \( k \neq i, j \).
	That is, a two coordinate update achieves the maximum.

	First, observe that \cref{eq:variational-smoothness} is a convex maximization problem over a
	(convex) polyhedron.
	As a result, at least one solution
	occurs at an extreme point of the constraint set,
	\[ \calD = \cbr{d : \abr{d, 1} = 0, \norm{d}_1 \leq 1}. \]
	The proof proceeds by showing that all extreme points of \( \calD \) contain exactly
	two non-zero entries.
	Let \( d_e \) be any extreme point of \( \calD \) and suppose by way
	of contradiction that \( d_e \) has at least three non-zero entries.
	Denote these entries as \( d_1, d_2, d_3 \).
	Since at least one entry of \( d_e \) must be negative and one must be positive,
	we may assume without loss of generality that \( d_1, d_2 > 0 \) and
	\( d_3 < 0 \).

	Let \( \epsilon > 0 \) and define \( d_e' = d_e + e_1 \epsilon - e_2 \epsilon \).
	For \( \epsilon  \) sufficiently small it holds that
	\( d_1 + \epsilon > 0 \) and \( d_2 - \epsilon > 0 \) so that
	\[
		(d_1 + \epsilon) + (d_2 - \epsilon) + d_3 = d_1 + d_2 + d_3.
	\]
	We conclude
	\begin{align*}
		\norm{d'_e}_1
		 & = |d_1 + \epsilon| + |d_2 - \epsilon| + |d_3| \\
		 & = (d_1 + \epsilon) + (d_2 - \epsilon) + |d_3| \\
		 & = |d_1| + |d_2| + |d_3|                       \\
		 & = d_1 + d_2 + d_3                             \\
		 & = \norm{d_e}_1                                \\
		 & \leq 1.
	\end{align*}
	Thus, \( d_e' \in \calD \).
	Define \( d_e'' = d_e - e_1 \epsilon + e_2 \epsilon \)
	and observe
	\( d_e'' \in \calD \) by a symmetric argument.
	Moreover,
	\[
		d_e = \half d_e' + \half d_e'',
	\]
	i.e. the extreme point is a convex combination of two points in \( \calD \).
	This contradicts the definition of an extreme point, so we have proved
	that every extreme point of \( \calD \) has at most two non-zero entries
	Since no point of \( \calD \) can have exactly one non-zero entry
	and \( 0 \) is the relative interior of \( \calD \), we have shown every extreme
	point has exactly two non-zero entries.

	As a result, \eqref{eq:variational-smoothness}
	is maximized at at least one extreme point \( d_e \),
	where \( \text{supp}(d_e) = 2 \).
	Thus, we may restrict optimization to directions of support two,
	giving
	\begin{align*}
		 & \max_{d} \cbr{ d^\top \nabla^2 f(x)d : \abr{d, 1} = 0, \|d\|_1 \leq 1 }                  \\
		 & \quad \quad =
		\max_{d} \cbr{ d^\top \nabla^2 f(x)d : \abr{d, 1} = 0, \text{supp}(d) = 2, \|d\|_1 \leq 1 } \\
		 & \quad \quad \leq L_1.
	\end{align*}
	It is now straightforward to obtain the final result using a Taylor expansion
	and the Lagrange form of the remainder.
	In particular, for some parameter \( x' \text{Conv}(\cbr{x, x + d}) \) we have
	\begin{align*}
		f(x & + d)
		= f(x) + \abr{\nabla f(x), d} + \frac{1}{2} d^\top \nabla^2 f(x + \alpha d) d \\
		    & \leq f(x) + \abr{\nabla f(x), d}
		+ \frac{1}{2} \norm{d}_1^2 \max_{v}  \cbr{
			v^\top \nabla^2 f(x') v
			: \abr{v, 1} = 0, \norm{v}_1 \leq 1
		}                                                                             \\
		    & = f(x) + \abr{\nabla f(x), d}
		+ \frac{1}{2} \norm{d}_1^2 \max_{v}  \cbr{
			v^\top \nabla^2 f(x') v
			: \abr{v, 1} = 0, \text{supp}(v) = 2, \norm{v}_1 \leq 1
		}                                                                             \\
		    & = f(x) + \abr{\nabla f(x), d} + \frac{L_1}{2} \norm{d}_1^2,
	\end{align*}
	which gives the result.
\end{proof}

\begin{proposition}\label{prop:L12}
	The constant \( L_1 \) in~\eqref{eq:ell1-smoothness} is exactly equal to \( \frac{L_2}{2} \).
\end{proposition}
\begin{proof}
	Let \( d \in \R^n \) such that \( \text{supp}(d) = 2 \) and
	\( \abr{d, 1} = 0 \).
	WLOG, suppose that the two non-zero entries of \( d \) are \( d_1 \) and \( d_2 \).
	Observe that \( \abr{d, 1} = 0 \) implies \( d_1 = - d_2 \)
	and \( \norm{d}_1 = \sqrt{2}\norm{d}_2 \). Thus we have
	\begin{align*}
		L_2
		 & = \sup_{x : \abr{x, 1} = a} \max_{d} \cbr{ d^\top \nabla^2 f(x)d : \abr{d, 1} = 0, \text{supp}(d) = 2 , \|d\|_2 \leq 1 } \\
		 & =
		2 \sup_{x : \abr{x, 1} = a} \max_{d} \cbr{ d^\top \nabla^2 f(x)d : \abr{d, 1} = 0, \text{supp}(d) = 2 , \|d\|_1 \leq 1 }    \\
		 & =
		2 L_1,
	\end{align*}
	where we have used \cref{prop:second-order-def} to relate the variational
	characterizations to the Lipschitz constants in question.
	This completes the proof.

\end{proof}

\begin{proposition}\label{prop:second-order-def}
	Let \( \norm{\cdot} \) an arbitrary norm and define the dual norm on the feasible space,
	\[
		\norm{v}_* = \sup \cbr{ z^\top v : \abr{z, 1} = 0, \text{supp}(z) = 2, \norm{z} \leq 1}.
	\]
	Then the variational characterization based on the Hessian,
	\[
		L =
		\sup_{x : \abr{x, 1} = a } \max_{d} \cbr{ d^\top \nabla^2 f(x)d : \abr{d, 1} = 0, \text{supp}(d) = 2 , \|d\| \leq 1 },
	\]
	gives the two-coordinate Lipschitz constant
	of \( \nabla f \) (see \cref{eq:L2-smooth}) in norm \( \norm{\cdot} \)
	on the feasible space.
\end{proposition}
\begin{proof}
	Let \( x \) be feasible (i.e. \( \abr{x, 1} = a \)) and define
	\[ \calD = \cbr{d : \abr{d, 1} = 0, \text{supp}(d) = 2, \norm{d}_1 \leq 1}. \]
	Suppose \( d \) is some be feasible 2-coordinate update, not necessarily unit norm.
	The fundamental theorem of calculus implies
	\[
		\nabla_{ij}f(x + d) - \nabla_{ij}f(x) = \int_{0}^1 \nabla_{ij}^2 f(x + t d) d dt
	\]
	Taking norms on both sides, we obtain
	\begin{align*}
		\norm{\nabla_{ij}f(x + d) - \nabla_{ij}f(x)}_*
		 & = \norm{\int_{0}^1 \nabla_{ij}^2 f(x + t d) d dt}_*                                     \\
		 & \leq \int_0^1 \norm{\nabla_{ij}^2 f(x') d }_* dt                                        \\
		 & \leq \norm{d} \int_0^1 \sup_{d' \in \calD} \cbr{ d'^\top \nabla_{ij}^2 f(x + td) d'} dt \\
		 & \leq L \norm{d},
	\end{align*}
	where we have used the definition of the dual norm.
	For the reverse inequality, let \( \tilde L \) be the Lipschitz constant
	of \( \nabla f \) in norm \( \norm{\cdot} \).
	Observe that for any feasible \( x \)
	and 2-coordinate update \( d  \), there exists \( \alpha \in (0,1) \)
	and \( \tilde x = x + \alpha d \) such that
	\[
		\nabla_{ij}^2 f(\tilde x) d = \nabla_{ij} f(x + d) - \nabla_{ij} f(x).
	\]
	Using this, we obtain
	\begin{align*}
		d^\top \nabla_{ij}^2 f(\tilde x) d
		 & \leq \norm{d} \norm{\nabla_{ij}^2 f(\tilde x) d}_*           \\
		 & = \norm{d} \norm{\nabla_{ij} f(x + d) - \nabla_{ij} f(x)}_*  \\
		 & \leq \tilde L \norm{d}^2                                   .
	\end{align*}
	Dividing by sides by \( \norm{d}^2 \), taking \( \norm{d} \rightarrow 0 \), and supremizing over \( x \), \( d \)
	gives
	\[
		L = \sup_{x : \abr{x, 1} = a } \max_{d \in \calD} \cbr{ d^\top \nabla^2 f(x)d} \leq \tilde L
	\]
	We conclude \( \tilde L = L \) as desired.
\end{proof}

\section{Relationship Between Proximal-PL Constants}
\label{proxpl1norm}

\begin{lemma}
	Suppose that \( F(x) = f(x) + g(x) \) satisfies the proximal-PL inequality
	in the \( \ell_2 \)-norm with constants \( L_2, \mu_2 \).
	Then \( F \) also satisfies the proximal-PL inequality in the \( \ell_1 \)-norm
	with constants \( L_1 \) and \( \mu_1 \in \sbr{\mu_2 / n, \mu_2} \).
\end{lemma}
\begin{proof}
	Proximal-PL inequality in the \( \ell_2 \)-norm implies
	\begin{align*}
		F(x) - F(x^*)
		 & \leq -\frac{L_2}{\mu_2} \min_{y} \cbr{\abr{\nabla f(x), y - x} + \frac{L_2}{2} \norm{y - x}_2^2 + g(y) - g(x)}          \\
		 & \leq -\frac{L_2}{\mu_2} \min_{y} \cbr{\abr{\nabla f(x), y - x} + \frac{L_2}{2n} \norm{y - x}_1^2 + g(y) - g(x)}         \\
		 & \leq -\frac{L_2 L_1 n}{L_2\mu_2} \min_{y} \cbr{\abr{\nabla f(x), y - x} + \frac{L_1}{2} \norm{y - x}_1^2 + g(y) - g(x)} \\
		 & = -\frac{L_1 n}{\mu_2} \min_{y} \cbr{\abr{\nabla f(x), y - x} + \frac{L_1}{2} \norm{y - x}_1^2 + g(y) - g(x)},
	\end{align*}
	where the last inequality follows from \citet{karimireddy2018adaptive}[Lemma 9] with the choice of
	\( \beta = \frac{L_2}{L_1 n} \), \( h(y) = \abr{\nabla f(x), y - x} + g(y) - g(x) \), and \( V(y) = \sqrt{L_2/2n}\norm{y - x}_1 \).
	Note that \( \beta \in (0, 1] \) since \( L_1 n \geq L_2 \)
	and \( h(x) = V(x) = 0 \) so that the conditions of the lemma are satisfied.
	We conclude that proximal-PL inequality holds with \( \mu_1 \geq \mu_2 / n \).

	We establish the reverse direction similarly; starting from proximal-PL
	in the \( \ell_1 \)-norm,
	\begin{align*}
		F(x) - F(x^*)
		 & \leq -\frac{L_1}{\mu_1} \min_{y} \cbr{\abr{\nabla f(x), y - x} + \frac{L_1}{2} \norm{y - x}_1^2 + g(y) - g(x)}        \\
		 & \leq -\frac{L_1}{\mu_1} \min_{y} \cbr{\abr{\nabla f(x), y - x} + \frac{L_1}{2} \norm{y - x}_2^2 + g(y) - g(x)}        \\
		 & \leq -\frac{L_1 L_2}{L_1\mu_1} \min_{y} \cbr{\abr{\nabla f(x), y - x} + \frac{L_2}{2} \norm{y - x}_2^2 + g(y) - g(x)} \\
		 & = -\frac{L_2}{\mu_1} \min_{y} \cbr{\abr{\nabla f(x), y - x} + \frac{L_2}{2} \norm{y - x}_2^2 + g(y) - g(x)},
	\end{align*}
	where now we have used the same lemma with $V(y) = \sqrt{L_1/2}\norm{y-x}_2$ and  \( \beta = \frac{L_1}{L_2} \), noting that \( \beta \in (0, 1] \)
	since \( L_1 \leq L_2 \).
	This shows that \( \mu_2 \geq \mu_1 \), which completes the proof.
\end{proof}

\section{Analysis of GS-q for Bound-Constrained Problem}\label{app:GSq}

In this section, we show linear convergence of greedy 2-coordinate descent
under a linear equality constraint and bound constraints for the problem
in~\eqref{eq:bounds} when using the GS-q rule.
First, we introduce two definitions which underpin the theoretical machinery
used in this section.
\begin{definition}[Conformal Vectors]
	Let \( d, d' \in \R^n \). We say that \( d' \) is conformal to \( d \) if
	\[
		\cbr{i : d_i' \neq 0} \subseteq \cbr{i : d_i \neq 0},
	\]
	that is, the support of \( d' \) is a subset of the support of \( d \),
	and \( d_i d_i' \geq 0 \) for every \( i \in \cbr{1, \ldots n} \).
\end{definition}

\begin{definition}[Elementary Vector]
	Let \( S \subset \R^n \) be a subspace. A vector \( d \in \calS \) is
	an elementary vector of \( \calS \) if there does not exist \( d' \)
	conformal to \( d \) with strictly smaller support, that is
	\[
		\cbr{i : d_i' \neq 0} \subsetneq \cbr{i : d_i \neq 0}.
	\]
\end{definition}
With these definitions in hand, we can state
\cref{lemma:non_elementary_realizations}, which is the key property we use in
our proof strategy.
\begin{lemma}[Conformal Realizations]
	\label{lemma:non_elementary_realizations}
	Let $S$ be a subspace of $\R^n$ and $t = \min_{x \in S} \supp(x)$.
	Let $\tau \in \{t, \dots, n\}$. Then every non-zero vector $x$ of $S \subseteq \R^n$ can be realized as the sum
	\[ x = d_1 + \dots + d_{s} + d_{s+1}, \]
	where $d_1, \dots, d_{s}$ are elementary vectors of $S$ that are conformal to $x$ and $d_{s+1} \in S$ is a vector conformal to $x$ with $\supp(d_{s+1}) = \tau$.  Furthermore, $s \leq n - \tau$.
\end{lemma}
We include a proof in Appendix \ref{appendix:conformal_realizations};
see \citet[Proposition~6.1]{tseng2009block} for an alternative (earlier) statement and proof.
Using this tool, we prove the following convergence rate for 2-coordinate
descent with the GS-q rule.
\begin{theorem}
	\label{theorem:convergence_box}
	Let the function $F(x)= f(x)+ h(x)$, where $f: \mathbb{R}^n \xrightarrow{}  \mathbb{R}$ is a smooth function and $h(x)$ is the box constraint indicator,
	\begin{align*}
		h(x) =  \begin{cases}
			        0      & \mbox{if $l_i \leq  x_i \leq u_i$ for all $i \in \{1, \dots, n\}$} \\
			        \infty & \mbox{otherwise}
		        \end{cases}
	\end{align*}
	Assume that \( F \) satisfies the proximal-PL condition in the \
	2-norm with constant constant $\mu_2$ and that
	$f$ is 2-coordinate-wise Lipschitz in the 2-norm.
	Then, minimizing
	\begin{align*}
		 & \min_{x\in\R^n} f(x), \nonumber \hspace{0.3 cm} \\\text{subject to     } & \abr{x, 1} = \gamma, \, x_i \in [l_i, u_i]
	\end{align*}
	using 2-coordinate descent with coordinate blocks selected according
	to the GS-q rule obtains the following linear rate of convergence:
	\[
		f(x^k) - f^* \leq \parens{1 - \frac{\mu_2}{L_{2}(n - 1)}}^k \parens{f(x^0) - f^*}.
	\]
\end{theorem}
We provide the proof in \cref{proof:convergence_box}.
The proof instantiates a more general result which holds for arbitrary
functions \( h \) and larger blocks sizes.

\subsection{Proof of Lemma~\ref{lemma:non_elementary_realizations}}
\label{appendix:conformal_realizations}
\begin{proof}
	The proof extends \citet[Proposition~9.22]{bertsekas1998network}. Consider $x \in S$. If $\supp(x) = \tau$, then let $d_1 = x$ and we are done.
	Otherwise, by Lemma \ref{lemma:elementary_existence} there exists an elementary vector $d_1 \in S$ that is conformal to $x$. Let
	\begin{align*}
		\gamma & = \max \bigg\{ \gamma \ \biggmid \
		[x]_j - \gamma [d_1]_j \geq 0 \quad \forall j \text{ with } [x]_j > 0 \quad \text{and} \\
		       & \hspace{4cm}
		[x]_j - \gamma [d_1]_j \leq 0 \quad \forall j \text{ with } [x]_j < 0. \bigg\}.        \\
	\end{align*}
	The vector $\gamma d_1$ is conformal to $x$. Let $\bar x = x - \gamma d_1$.  If $\supp(x_1) \leq \tau$, choose $d_2 = \bar x$ and we are done. Note that $d_2 \in S$ since $S$ is closed under subtraction. Otherwise, let $x = \bar x$ and repeat the process. Let $s$ be the number of times this process is conducted. Each iteration reduces the number of non-zero coordinates of $x$ by at least one. Since it terminates when $\supp(x) = \tau$, we have $s \leq n - \tau$.\\
\end{proof}

\subsection{Proof of Theorem~\ref{theorem:convergence_box}}\label{proof:convergence_box}

We prove the result by instantiating a more general convergence theorem
for optimization with linear constraints \( A x = c \),
where \( A \in \R^{m \times n} \), and general non-smooth regularizers
\( h \).
We assume \( A \) is full row-rank and that the proximal operator for \( h \)
is easily computed.
Note that, in this setting, block coordinate descent must operates on blocks
\( b_i \subset [n] \) of size \( m + 1 \leq \tau \leq n \) in order to maintain
feasibility of the iterates.
Let \( U_{b_i}(d_{b_i}) \) map block update vector \( d_{b_i} \)
from \( \R^{\tau} \) to \( \R^n \) by augmenting it with zeros
and define
\[
	S_{b_i} = \cbr{ d_{b_i} : A U_{b_i}(d_{b_i}) = 0}.
\]
That is, \( S_{b_i} \) is the null space of \( A \)  overlapping with block
\( b_i \).

As mentioned before, the notions of conformal and elementary vectors
introduced in the previous section provide necessary tools for our
convergence proof.
The following Lemmas provide the main show the utility of these definitions
for optimization.

\begin{lemma}[{\citet[Lemma~2]{necoara2014random}}]
	\label{lemma:support_of_elementary_vectors}
	Given $d \in \Null(A)$, if $d$ is an elementary vector of $\Null(A)$, then
	\[ \supp(d) \leq \text{rank}(A) + 1. \]
\end{lemma}

\begin{lemma}[{\citet[Proposition~9.22]{bertsekas1998network}}]\label{lemma:elementary_existence}
	Let \( S \) be a subspace of $\R^n$.
	Then vector $d \in S$ is either a elementary vector of $S$, or there
	exists an elementary vector $d' \in S$ that is conformal to $d$.\\
\end{lemma}

\begin{lemma}[{\citet[Lemma~6.1]{tseng2009block}}]
	\label{lemma:separable_inequality}
	Let $h$ be a coordinate-wise separable and convex function.
	For any $x$, $x + d \in \text{dom}(h)$, let $d$ be expressed as
	$d = d_1 + \dots + d_s$ for some $s \geq 1$ and some
	non-zero $d_t \in \R^n$ conformal to $d$ for $t = 1, \dots, s$. Then
	\[ h(x + d) - h(x) \geq \sum_{t=1}^s \parens{h \parens{x + d_t} - h(x)}. \]
\end{lemma}

We are now ready to prove our general convergence result for block-coordinate
descent with linear constraints and the GS-q block selection rule.
We emphasize that in the following theorem:
(i) \( h \) need not be the indicator for box constraints;
(ii) \( A \) many consist of many coupling constraints; and
(iii) the convergence rate improves with block-size \( \tau \), unlike many similar
results.

\begin{proposition}\label{prop:gs-q-general-rate}
	Fix block size \( \tau \geq m + 1 \) and let \( \calB \) be the set of
	all blocks \( b_i \subset [n] \) of size \( \tau \).
	Consider solving the linearly constrained problem
	\begin{align*}
		\min_{x\in\R^n} F(x)   & := f(x) + h(x), \\
		\text{subject to     } & A x  = c
	\end{align*}
	where the gradient of \( f \) is \( \tau \)-coordinate
	Lipschitz with constant \( L_2 \) and \( h \) is convex and coordinate-wise
	separable.
	Suppose \( F \) satisfies the proximal-PL inequality in the \( 2 \)-norm
	with constant \( \mu_2 \).
	Then the block-coordinate descent method with blocks given by the GS-q
	rule converges as
	\[
		F(x^k) - F^* \leq \rbr{1 - \frac{\mu_2}{L_2(n - \tau + 1)}}^k \rbr{F(x^0) - F^*}.
	\]
\end{proposition}
\begin{proof}
	Block-coordinate Lipschitz continuity of $\nabla f$ give the following version of the descent lemma:
	\begin{align*}
		f(x^{k+1})                          & \leq f(x^k) + \langle \nabla f(x^k), x^{k+1} - x_{k} \rangle + \frac{L_2}{2}\norm{x^{k+1} - x^k}_2^2                                                      \\
		\intertext{We have $x^{k+1} =  x^k + U_{b^k}(d^*_{b^k})$ by definition of the update rule. Substituting this into the descent lemma gives}
		f(x^{k+1})                          & \leq f(x^k) + \langle \nabla_{b^k} f(x^k), d^*_{b^k} \rangle + \frac{L_2}{2}\norm{d^*_{b_i^k}}_2^2                                                        \\
		\Rightarrow f(x^{k+1}) + h(x^{k+1}) & \leq f(x^k) + \langle \nabla_{b^k} f(x^k), d_{b^k} \rangle + \frac{L_2}{2}\norm{d^*_{b_i^k}}_2^2 + h(x^{k+1}) + h(x^k) - h(x^k)                           \\
		\Rightarrow  F(x^{k+1})             & \leq F(x^k) + \langle \nabla_{b^k} f(x^k), d^*_{b^k} \rangle + \frac{L_2}{2}\norm{d^*_{b_i^k}}_2^2 + h_{b^k}(x_{b^k}^k + d^*_{b^k}) - h_{b^k}(x_{b^k}^k).
	\end{align*}
	Substituting in the choice of coordinate block $b^k$ according to the GS-q rule and the definition of $d^*_{b^k}$ gives
	\begin{align*}
		F(x^{k+1}) & \leq F(x^k) + \min_{b_i \in B} \bigg\{  \min_{d_{b_i} \in S_{b_i}} \big\{ \langle \nabla_{b_i} f(x^k), d_{b_i} \rangle + \frac{L_2}{2}\norm{d^*_{b_i^k}}_2^2 \\
		           & \hspace{5cm} + h_{b_i}(x^k_{b_i} + d_{b_i}) - h_{b_i}(x^k_{b_i}) \big\}\bigg\}.
	\end{align*}
	For clarity, we define the quadratic upper bound to be the function
    \begin{equation}
	\begin{aligned}
		V(x^k, d_{b_i}) & = \langle \nabla_{b_i} f(x^k), d_{b_i} \rangle + \frac{L_2}{2}\norm{d^*_{b_i^k}}_2^2 + h_{b_i}(x^k_{b_i} + d_{b_i}) - h_{b_i}(x^k_{b_i}), \\
		\text{which gives,}\\
		F(x^{k+1})      & \leq F(x^k) + \min_{b_i \in B} \curly{  \min_{d_{b_i} \in S_{b_i}} \curly{V(x^k, d_{b_i})}}. \label{eq:greedy_step}
	\end{aligned}
 \end{equation}
	We must control that the right-hand-side of \eqref{eq:greedy_step} in terms of the full-coordinate minimizer
	\[
		d^* = \argmin{d \in \Null(A)} \curly{ \langle \nabla f(x^k), d \rangle + \frac{L_2}{2}\norm{d}_2^2 + h(x^k + d) - h(x^k) }.
	\]
	in order to apply the prox-PL inequality. We briefly digress and consider conformal realizations of $d^*$ in order to do so.\\

	By lemma \ref{lemma:non_elementary_realizations}, $d^*$ has a conformal realization
	\[ d^* = d_1^* + \dots + d_r^* + d_{r+1}^*,  \]
	where $r \leq n - \tau$ and $d_1^*, \dots d_r^*$ are elementary vectors of $\Null(A)$ and $d_{r+1}^* \in \Null(A)$.
	Lemma \ref{lemma:support_of_elementary_vectors} gives $\text{supp}(d_l^*) \leq m+1$; therefore there exists $ b_i \in B$ such that $d_l^* \in S_{b_i}$ for all $l = 1, \dots, r$.
	By construction, $\supp(d_{r+1}^*) = \tau$ and so there also exists $ b_i \in B $ such that $d_{r+1}^* \in S_{b_i}$.
	Let $\bar B \subseteq B$ be the smallest set of blocks such that
	\[  \forall \ l \in \curly{1, \dots, r+1 }, \  \exists b_i \in \bar B, \quad  d_l^* \in S_{b_i}, \]
	and observe that $|\bar B| \leq n - 1$.\\

	Returning to \eqref{eq:greedy_step}, we can use the fact that the value of $V(x^k, d_{j})$ obtained at the minimizing block $b^{k} \in B$ is less than or equal to the average over the subset of blocks $\bar B$:
	\begin{align}
		\label{eq:min_inequality}
		\min_{b_i \in B} \curly{  \min_{d_{b_i} \in S_{b_i}} \curly{ V(x^k, d_{b_i}) }} & \leq \frac{1}{|\bar B|} \sum_{b_i \in \bar B} \min_{d_{b_i} \in S_{b_i}} \curly{ V(x^k, d_{b_i}) }.
	\end{align}
	Combining this result with \eqref{eq:greedy_step} and \eqref{eq:min_inequality}, we obtain
 \begin{equation}
	\begin{aligned}
		F(x^{k+1}) & \leq F(x^k) + \frac{1}{|\bar B|} \sum_{b_i \in \bar B} \min_{d_{b_i} \in S_{b_i}} \curly{ V(x^k, d_{b_i}) }                                                                                                  \\
		           & = F(x^k) + \frac{1}{|\bar B|} \min_{d_{b_i} \in S_{b_i}, \forall b_i \in \bar B } \curly{ \sum_{b_i \in \bar B} V(x^k, d_{b_i}) }                                 \\
		           & = F(x^k) + \frac{1}{|\bar B|} \min_{d_{b_i} \in S_{b_i}, \forall b_i \in \bar B } \bigg\{ \langle \nabla f(x^k), \sum_{b_i \in \bar B} d_{b_i} \rangle \\
             &\hspace{1cm}+ \sum_{b_i \in \bar B} \frac{L_2}{2}\norm{d_{b_i}}^2 + \sum_{b_i \in \bar B} \parens{h_{b_i}(x^k_{b_i} + d_{b_i}) - h_{b_i}(x^k_{b_i}) } \bigg\}. \label{eq:current_upper_bound}  \\ 
             \end{aligned}
        \end{equation}
		For all $b_i \in \bar B$, substituting any $d_{b_i} \in S_{b_i}$ for the vector in $S_{b_i}$ that minimizes \eqref{eq:current_upper_bound} can only increase the upper bound. Choosing the $d^*_l$ corresponding to each block $b_i \in \bar B$ yields
  \begin{align*}
		        F(x^{k+1})   & \leq F(x^k) + \frac{1}{|\bar B|} \bigg( \langle \nabla f(x^k), \sum_{l=1}^{r+1} d_l^* \rangle + \sum_{l=1}^{r+1} \frac{L_2}{2}\norm{d_l^*}^2                                                              \\
		           & \hspace{5cm} + \sum_{l=1}^{r+1} \parens{h_{b_i}(x^k_{b_i} + d_l^*) - h_{b_i}(x^k_{b_i}) }\bigg).
 \end{align*}
	We now use $d^* = \sum_{l=1}^{r+1} d_l^*$ and apply lemma \ref{lemma:separable_inequality} twice to obtain
 \begin{equation}
	\begin{aligned}
		F(x^{k+1}) & \leq F(x^k) + \frac{1}{|\bar B|} \bigg( \langle \nabla f(x^k), \sum_{l=1}^{r+1} d_l^* \rangle + \frac{L_2}{2} \norm{d^*}^2                                     \\
		           & \hspace{5cm} + \sum_{l=1}^{r+1} \parens{h_{b_i}(x^k_{b_i} + d_l^*) - h_{b_i}(x^k_{b_i}) }\bigg)                                                             \\
		F(x^{k+1}) & \leq F(x^k) + \frac{1}{|\bar B|} \curly{ \langle \nabla f(x^k), d^* \rangle + \frac{L_2}{2} \norm{d^*}^2 + h(x^k + d^*) - h(x^k)}                             \\
		           & = F(x^k) + \frac{1}{|\bar B|} \min_{d \in S} \curly{ \langle \nabla f(x^k), d \rangle + \frac{L_2}{2} \norm{d}_2^2 + h(x^k + d) - h(x^k)}. \label{eq:pl_ineq}
	\end{aligned}
 \end{equation}
	Applying the prox-PL inequality in the $\norm{\cdot}_2$ norm gives
	\begin{align*}
		F(x^{k+1}) & \leq F(x^k) - \frac{\mu_2}{|\bar B|} (F(x^k) - F^*)       \\
		           & = F(x^k) - \frac{\mu_2}{L_2(n -  \tau + 1)} (F(x^k) - F^*).
	\end{align*}
	Subtracting $F^*$ from both sides and applying the inequality recursively completes the proof.
\end{proof}

Instantiating \cref{prop:gs-q-general-rate} with \( A = 1^\top \),
\( c = \gamma \), \( \tau = 2 \) and
\begin{align*}
	h(x) =  \begin{cases}
		        0      & \mbox{if $l_i \leq  x_i \leq u_i$ for all $i \in \{1, \dots, n\}$} \\
		        \infty & \mbox{otherwise}
	        \end{cases}
\end{align*}
is sufficient to obtain \cref{theorem:convergence_box}.

\section{Greedy Rules Depending on Coordinate-Wise Constants}
\label{app:Li}

We first derive the greedy GS-q rule, then steepest descent in the L-norm, and then give a dimension-independent convergence rate based on the L-norm.

\subsection{GS-q Rule with Coordinate-Wise Constants}

The GS-q rule under an equality constraint and coordinate-wise Lipschitz constants is given by
\begin{equation}
	\argmin\limits_b\left\{ \min_{d_b|d_i +d_j = 0} \langle \nabla_b f(x),d_b \rangle + \frac{L_i}{2}d_i^2+\frac{L_j}{2}d_j^2\right\}.
	\label{eq:Li GSQ}
\end{equation}
\paragraph{Solving for $d_b$.} We first fix $b$ and solve for $d_b$.
The Lagrangian of the inner minimization in~\eqref{eq:Li GSQ}is:
\begin{equation*}
	\mathcal{L}(d,\lambda) = \innerprod{\nabla_b f(x)}{d_b} +\frac{L_i}{2}d_i^2+\frac{L_j}{2}d_j^2+\lambda(d_i+d_j).
\end{equation*}
Set the gradient with respect $d_i$ to zero we get
\[
	\nabla_if(x)+L_id_i+\lambda =0,
\]
and solving for $d_i$ gives
\begin{align}
	d_i= \frac{-\nabla_i f(x)-\lambda}{L_i}. \label{eq:dilambda}
\end{align}
Similarly, we have
\[
	d_j= \frac{-\nabla_j f(x)-\lambda}{L_j}.
\]
Since $d_i = -d_j$ we have
\begin{equation*}
	\frac{-\nabla_i f(x)-\lambda}{L_i} = \frac{\nabla_j f(x)+\lambda}{L_j},
\end{equation*}
and solving for $\lambda$ gives
\begin{equation}
	\lambda =  \frac{-(L_j \nabla_i f(x) + L_i \nabla_j f(x))}{L_i+L_j}. \label{eq:dijlambda}
\end{equation}
Substituting \eqref{eq:dijlambda} in \eqref{eq:dilambda} gives
\begin{align*}
	d_i & = \frac{1}{L_i}\left(-\nabla_i f(x) - \frac{-(L_j \nabla_i f(x) + L_i \nabla_j f(x))}{L_i+L_j}\right)                   \\
	    & = \frac{1}{L_i}\left(\frac{-L_i\nabla_i f(x) - L_j\nabla_i f(x) + L_j\nabla_if(x) + L_i\nabla_j f(x)}{L_i + L_j}\right) \\
	    & = \frac{1}{L_i}\left(\frac{-L_i\nabla_i f(x) + L_i\nabla_j f(x)}{L_i + L_j}\right)                                      \\
	    & = -\frac{\nabla_i f(x) - \nabla_jf(x)}{L_i + L_j},
\end{align*}
and similarly
\[
	d_j = \frac{\nabla_i f(x) - \nabla_jf(x)}{L_i + L_j}.
\]
\paragraph{Solving for $b$.} We now use the optimal $d_i$ and $d_j$ in~\eqref{eq:Li GSQ},
\begin{align*}
	       & \argmin_b\left\{\nabla_if(x)d_i + \nabla_jf(x)d_j + \frac{L_i}{2}d_i^2 + \frac{L_j}{2}d_j^2\right\}                       \\
	\equiv & \argmin_b\left\{\nabla_if(x)d_i - \nabla_jf(x)d_i + \frac{L_i}{2}d_i^2 + \frac{L_j}{2}d_i^2\right\}                       \\
	\equiv & \argmin_b\left\{(\nabla_if(x)-\nabla_jf(x))d_i + \frac{L_i + L_j}{2}d_i^2\right\}                                         \\
	\equiv & \argmin_b\left\{-\frac{(\nabla_if(x)-\nabla_jf(x))^2}{L_i+L_j} + \frac{(\nabla_if(x)-\nabla_jf(x))^2}{2(L_i+L_j)}\right\} \\
	\equiv & \argmin_b\left\{-\half\frac{(\nabla_if(x)-\nabla_jf(x))^2}{L_i+L_j}\right\}                                               \\
	\equiv & \argmax_b\left\{\frac{(\nabla_if(x) - \nabla_jf(x))^2}{L_i+L_j}\right\}.
\end{align*}

\subsection{Steepest Descent with Coordinate-Wise Constants}
\label{sdL-norm}
Here, we show that steepest descent in the $L$-norm always admits at least one solution which updates only two coordinates. Steepest descent in the $L$-norm, subject to the equality constraint, takes steps in the direction $d$ that minimizes the following model of the objective:
\begin{equation}\label{eq:SDLnorm}
	d \in \argmin_{d \in \mathbb{R}^n | d^T 1 = 0} \left\{\nabla f(x)^T d+ \frac{1}{2\alpha} ||d||_L^2\right\},
\end{equation}
This is a convex optimization problem for which strong duality holds.
Introducing a dual variable $\lambda \in \R$, we obtain the Lagrangian
\begin{equation*}
	\mathcal{L}(d,\lambda) = \nabla f(x)^T d+ \frac{1}{2\alpha} ||d||_L^2 - \lambda(d^T 1).
\end{equation*}
The subdifferential with respect to $d$ and $\lambda$ yields necessary and sufficient optimality conditions for a steepest descent direction,
\begin{align*}
	\nabla_d \mathcal{L}(d,\lambda)         & = \nabla f(x) + \frac{1}{2\alpha} g - \lambda 1 = 0 \\ &\text{(for some subgradient $g \in \partial ||d||_L^2$)}\\
	\nabla_{\lambda} \mathcal{L}(d,\lambda) & = d^T 1 = 0.
\end{align*}
The second condition is simply feasibility of $d$, while from the first we obtain,
\begin{align}\label{eq:LiSDpartial}
	2\alpha(-\nabla f(x)+\lambda 1) \in \partial ||d||_L^2\nonumber \\
	\alpha(-\nabla f(x)+\lambda 1) \in ||d||_L (\sqrt{L} \odot \text{sgn}(d)),
\end{align}
where element $m$ of sgn$(d)$ is $1$ if $d_m$ is positive, $-1$ if $d_m$ is negative, and can be any value in $[-1,1]$ if $d_m$ is 0.
The following lemma shows that these conditions are always satisfied by a two-coordinate update.
\begin{lemma}\label{lemma:Lisd}
	Let $\alpha > 0$. Then at least one steepest descent direction with respect to the $1$-norm has exactly two non-zero coordinates.
	That is,
	\begin{align*}
		\min_{d \in \mathbb{R}^n | d^T 1 = 0} {\nabla f(x)^T d+ \frac{1}{2\alpha} ||d||_L^2}
		= \nonumber \\
		\min_{i,j} \cbr{\min_{d_{ij} \in \mathbb{R}^2 | d_{i} + d_j = 0} {\nabla_{ij} f(x)^T d_{ij}+ \frac{1}{2\alpha} ||d_{ij}||_L^2}}.
	\end{align*}
	\begin{proof}
		Similar to the steepest descent in the 1-norm, the proof follows by constructing a solution to the steepest descent problem in Eq.~\ref{eq:SDLnorm} which only has two non-zero entries. Let $i$ and $j$ maximize $(\nabla_i f(x) - \nabla_j f(x))/(\sqrt{L_i} + \sqrt{L_j})$.
		Our proposed solution is $d$ such that $d_i = -\delta, d_j = \delta$ for some $\delta \in \R$ and $d_{k,k \neq i,j}=0$. In order for this relationship in \eqref{eq:LiSDpartial} to hold, we would require
		\[
			-\alpha\nabla f(x)+\alpha\lambda 1 \in \|d\|_L (\sqrt{L}\odot sgn(d)).
		\]
		From the definition of L-norm and our definition of $d$ that
		\begin{align*}
			\|d\|_L & = \sqrt{L_i} \delta +\sqrt{L_j}\delta \\
			        & = \delta(\sqrt{L_i}+\sqrt{L_j}).
		\end{align*}
		Also, we know that sgn$(d_i)=-1$, sgn$(d_j)=1$, sgn$(d_k)=[-1,1]$. Therefore, we would need
		\begin{align}
			-\alpha\nabla_i f(x)+\alpha\lambda & = -\delta\sqrt{L_i}(\sqrt{L_i}+ \sqrt{L_j}) \label{eq:lnormcond1}     \\
			-\alpha\nabla_j f(x)+\alpha\lambda & = \delta\sqrt{L_j}(\sqrt{L_i}+ \sqrt{L_j})\label{eq:lnormcond2}       \\
			-\alpha\nabla_k f(x)+\alpha\lambda & = \delta\sqrt{L_k}(\sqrt{L_i}+ \sqrt{L_j})[-1,1]\label{eq:lnormcond3}
		\end{align}
		From \eqref{eq:lnormcond1}, $\lambda = \nabla_i f(x) -\frac{\delta}{\alpha}\sqrt{L_i}(\sqrt{L_i}+ \sqrt{L_j}) $. Substituting $\lambda$ in \eqref{eq:lnormcond2}, we get
		\begin{align*}
			-\alpha\nabla_j f(x)+ \alpha\nabla_i f(x)-
			\delta\sqrt{L_i}(\sqrt{L_i}+ \sqrt{L_j}) & =  \delta\sqrt{L_j}(\sqrt{L_i}+ \sqrt{L_j}) \\
			\alpha\nabla_i f(x) - \alpha\nabla_j f(x)            & =
			\delta(\sqrt{L_i}+ \sqrt{L_j})(\sqrt{L_i}+ \sqrt{L_j}),                                \\
		\end{align*}
		From this we get,
		\begin{equation}
			\delta = \frac{\alpha}{(\sqrt{L_i}+ \sqrt{L_j})^2}( \nabla_i f(x) - \nabla_j f(x)).
			\label{eq:L1iDelta}
		\end{equation}
		Using $\lambda$ in~\eqref{eq:lnormcond3} means that for variables $k\neq i$ and $k\neq j$ that we require
		\begin{align*}
			-\alpha\nabla_k f(x) + \alpha\nabla_if(x) - \delta\sqrt{L_i}(\sqrt{L_i}+\sqrt{L_j}) & \in \delta\sqrt{L_k}(\sqrt{L_i}+\sqrt{L_j})[-1,1]                \\
			-\alpha(\nabla_if(x) - \nabla_k f(x))                                         & \in \delta(\sqrt{L_i}+\sqrt{L_k})(\sqrt{L_i} + \sqrt{L_j})[-1,1] \\
			-\alpha\frac{\nabla_k f(x) - \nabla_if(x)}{(\sqrt{L_i}+\sqrt{L_k})}           & \in \delta(\sqrt{L_i} + \sqrt{L_j})[-1,1]
		\end{align*}
		Using the definition of $\delta$~\eqref{eq:L1iDelta} this is equivalent to
		\[
			-\frac{\nabla_i f(x) - \nabla_kf(x)}{\sqrt{L_i}+\sqrt{L_k}} \in \frac{\nabla_i f(x)-\nabla_jf(x)}{\sqrt{L_i} + \sqrt{L_j}}[-1,1],
		\]
		which holds due to the way we chose $i$ and $j$.

		We have shown that a two-coordinate update $d$ satisfies the sufficient conditions to be a steepest descent direction in the $L$-norm.
	\end{proof}
\end{lemma}
%

\subsection {Convergence result for coordinate-wise Lipschitz case}
\label{proof:Liresult}

Lemma~\ref{lemma:Lisd} allows us to give a dimension-independent convergence rate of a greedy 2-coordinate method that incorporates the coordinate-wise Lipschitz constants, by relating
the progress of the 2-coordinate update to the progress made by a full-coordinate steepest descent step. If we use $L_L$ as the Lipschitz-smoothness constant in the $L$-norm, then by the descent lemma we have
\[
f(x^{k+1}) \leq f(x^k) + \nabla f(x^k)^Td^k + \frac{L_L}{2}\norm{d^k}_L^2.
\]
From Lemma~\ref{lemma:Lisd}, if we use the greedy two-coordinate update to set $d^k$ and use a step size of $\alpha=1/L_L$ we have
\[
f(x^{k+1}) \leq f(x^k) + \min_{d |d^T1=0}\left\{\nabla f(x^k)^Td + \frac{L_L}{2}\norm{d}_L^2\right\}.
\]
Now subtracting $f^*$ from both sides and the proximal-PL assumption in the $L$-norm,
\begin{align*}
f(x^{k+1}) - f(x^*) & \leq f(x^k) - f(x^*)- \frac{1}{2L_L}\mathcal{D}(x^k,L_L)\\
 & = f(x^k) - f(x^*) - \frac{\mu_L}{L_L}(f(x^k) - f^*)  \\
 &= \left(1- \frac{\mu_L}{L_L}\right)(f(x^k) - f^*)
\end{align*}
It is possible to obtain a faster rate using a smallest setting of the $L_i$ such that $f$ is 1-Lipschitz in the $L$-norm. However, it is not obvious how to find such $L_i$ in practice.

\section{General Equality Constraints}

Rather a constraint of the form $\sum_i x_i = \gamma$, we could also consider general equality constraints of the form $\sum_i a_ix_i = \gamma$ for positive weights $a_i$. In this case the greedy rule would be
\[
	\argmax_{i,j}\left\{\frac{a_j\nabla_i f(x) - a_i\nabla_j f(x)}{a_1 + a_2}\right\},
\]
and we could use a $\delta^k$ of the form
\[
	\delta^k = -\frac{\alpha}{a_1+a_2}[a_2\nabla_1 f(w^k) - a_1\nabla_2 f(w^k)].
\]
Unfortunately, the greedy rule in this case appears to requirer $O(n^2)$. However, if re-parameterized in terms of variables $x_i/a_i$ then the constraint is transformed to $\sum_i x_i=\gamma$ and we can use the methods discussed in this work (although the ratio approximation also relies on re-parameterization so makes less sense here).

We could also consider the case performing greedy coordinate descent methods with a set of linear equality constraints. With $m$ constraints, we expect this to require updating $m+1$ variables. Although it is straightforward to define greedy rules for this setting, it is not obvious that they could be implemented efficiently.

\section{Additional Experiments}
In Figure~\ref{fig:randomSeeds}, we repeat the scaled version of our equality-constrained experiment with different seeds. We see that the performance across different random seeds is similar.
\begin{figure}[t]
	\centering
	\includegraphics[width=0.49\textwidth]{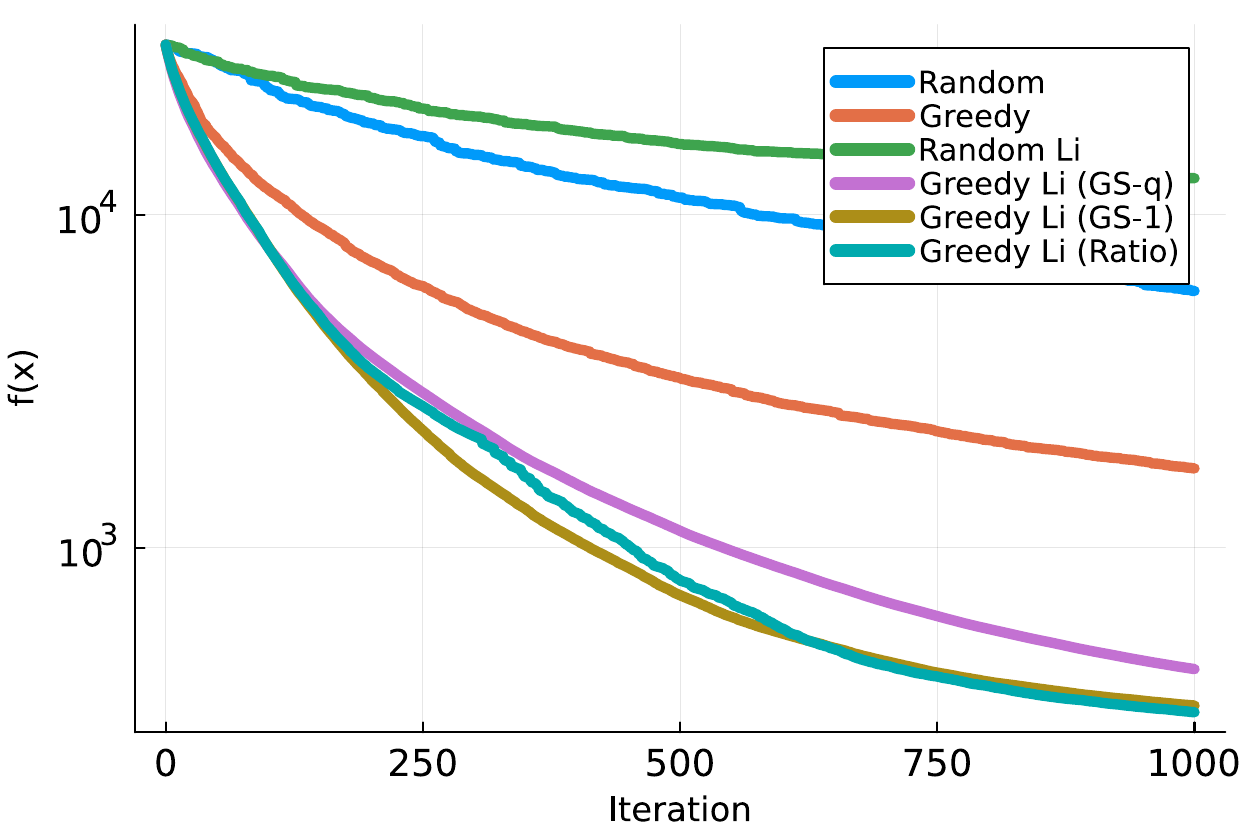}
	\includegraphics[width=0.49\textwidth]{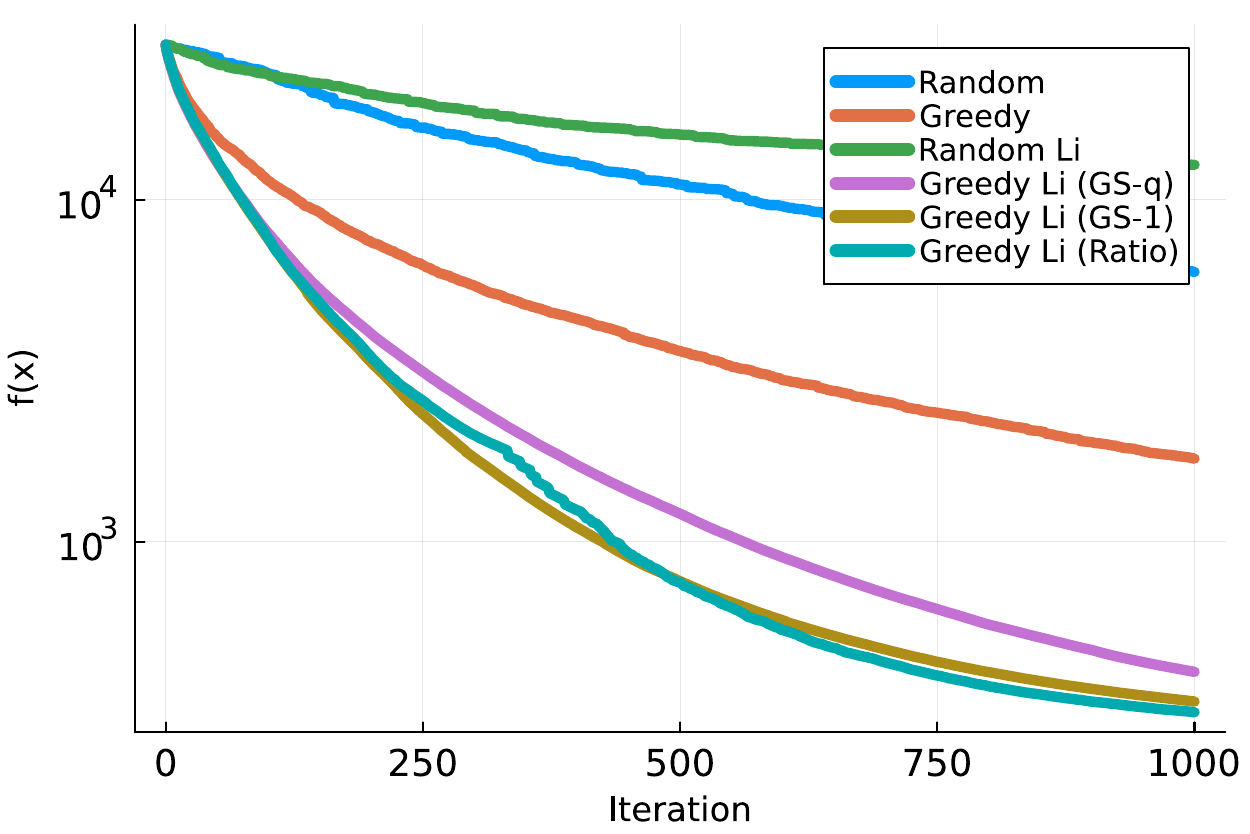}
	\includegraphics[width=0.49\textwidth]{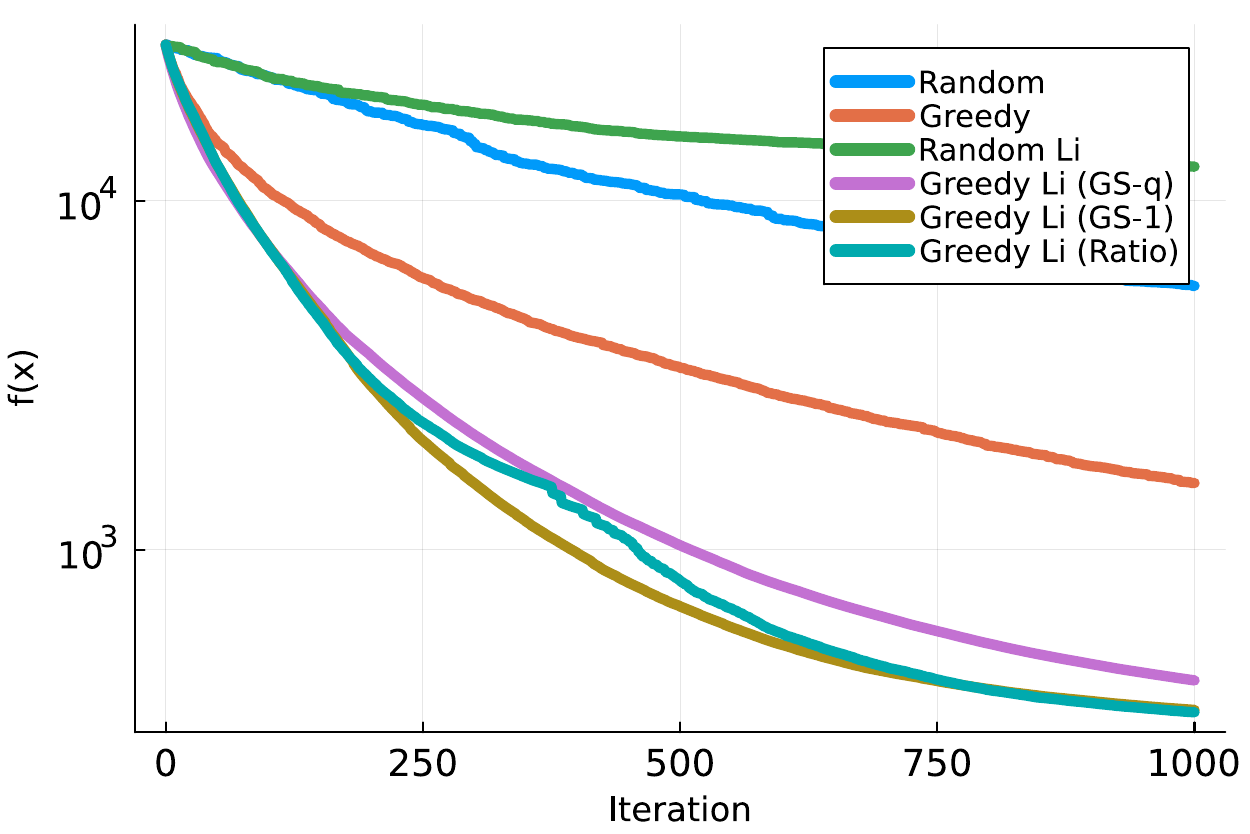}
	\includegraphics[width=0.49\textwidth]{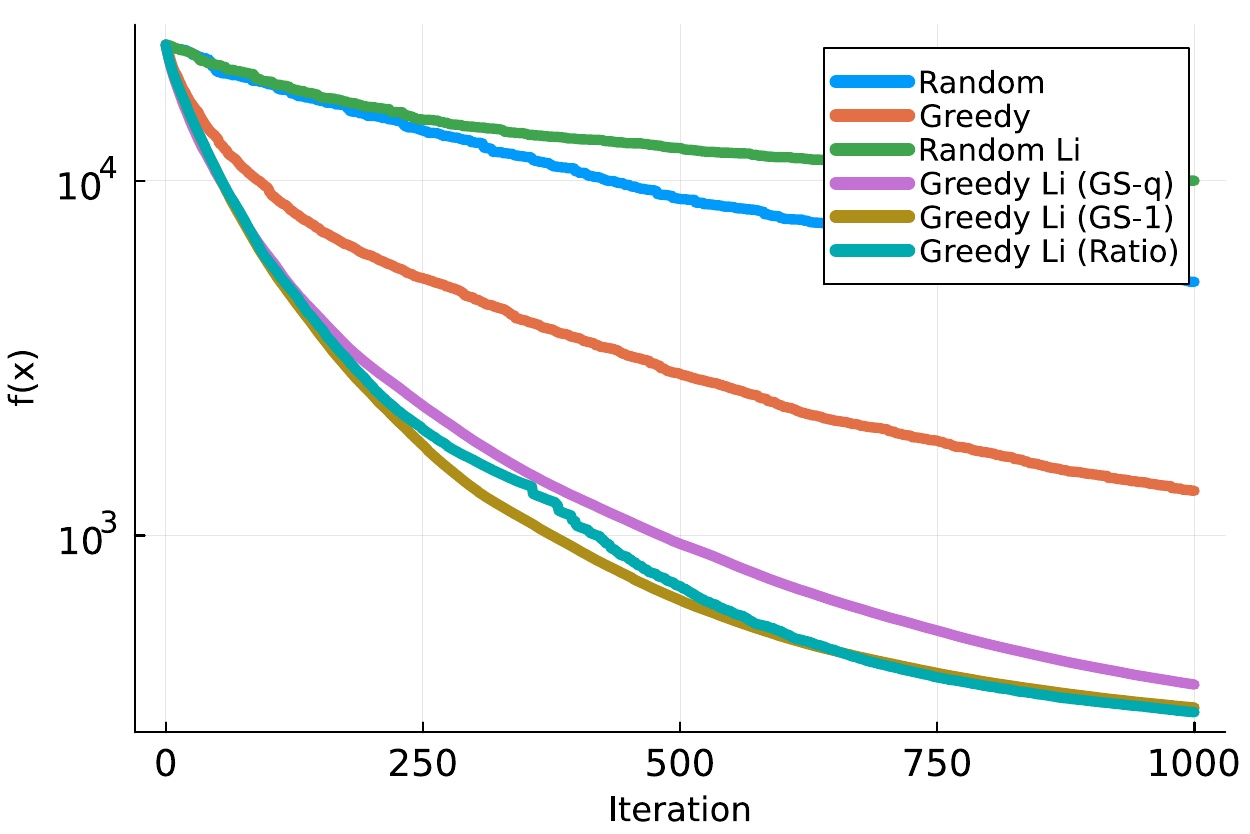}
	\caption{Comparison of different random and greedy rules under 4 choices for the random seed used to generate the data (and for the sampling in the randomized methods). }
	\label{fig:randomSeeds}
\end{figure}

We repeated the experiment that compares different greedy methods under equality and bound constraints with different seeds in Figures \ref{fig:boundconstdiffseed},~\ref{fig:intvardiffseed}, and~\ref{fig:numvardiffseed}. We see that the GS-q and GS-1 have a small but consistent advantage in terms of decreasing the objective while the GS-s and GS-1 rules have a consistent advantage in terms of moving variables to the boundaries. Finally, we see that the GS-1 rule only updates 2 variables on most iterations (over 85\%) while it updates 3 or fewer variables on all but a few iterations.
\begin{figure}
	\centering    \includegraphics[width=0.49\textwidth]{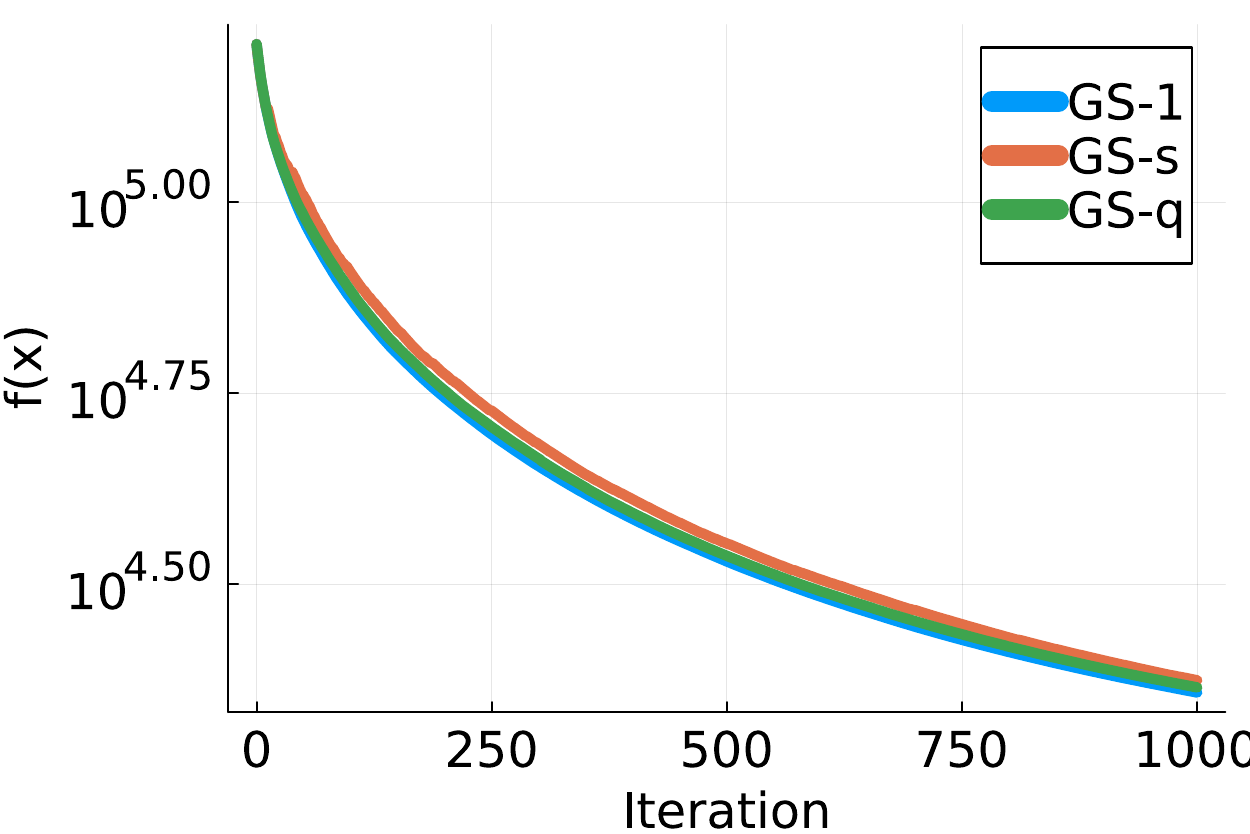}
	\includegraphics[width=0.49\textwidth]{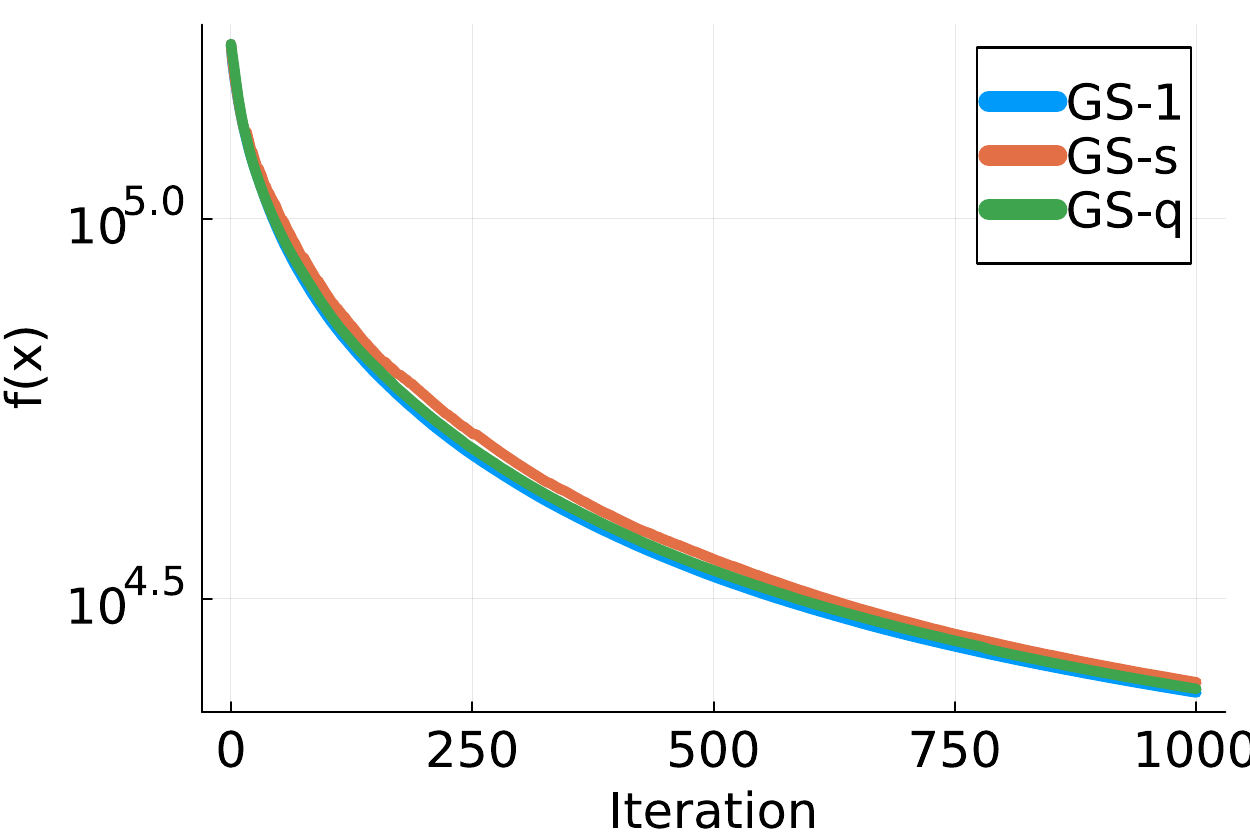}
	\includegraphics[width=0.49\textwidth]{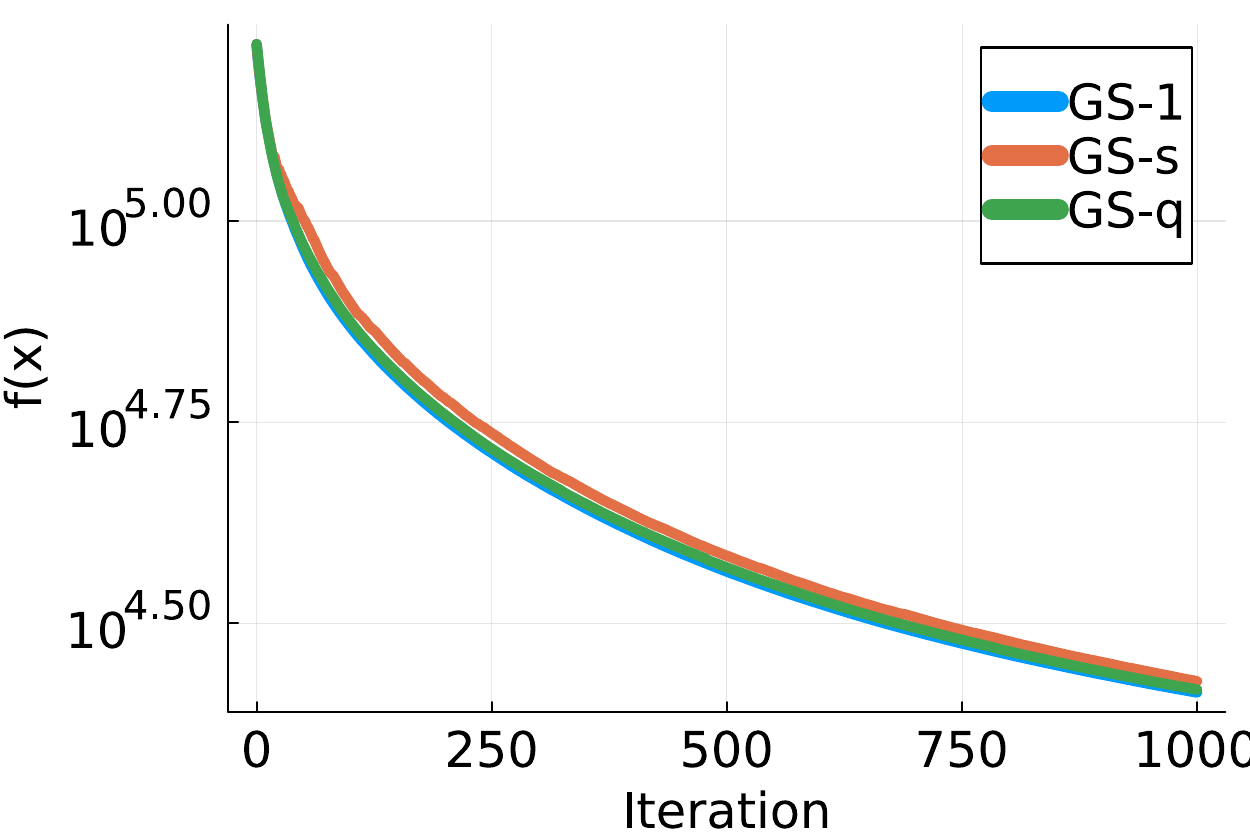}
	\includegraphics[width=0.49\textwidth]{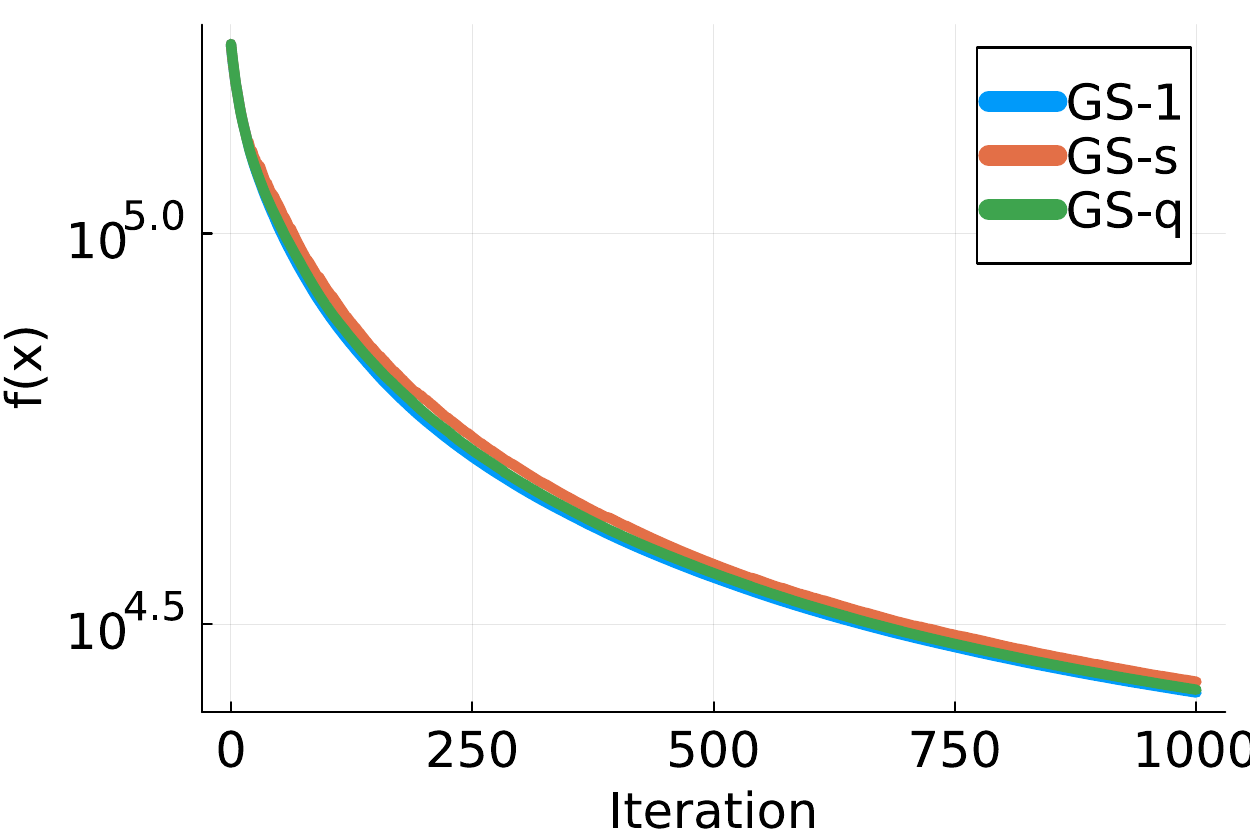}
	\caption{Comparison of different greedy rules under 4 choices for the random seed used to generate the data.}
	\label{fig:boundconstdiffseed}
\end{figure}
\begin{figure}
	\centering    \includegraphics[width=0.49\textwidth]{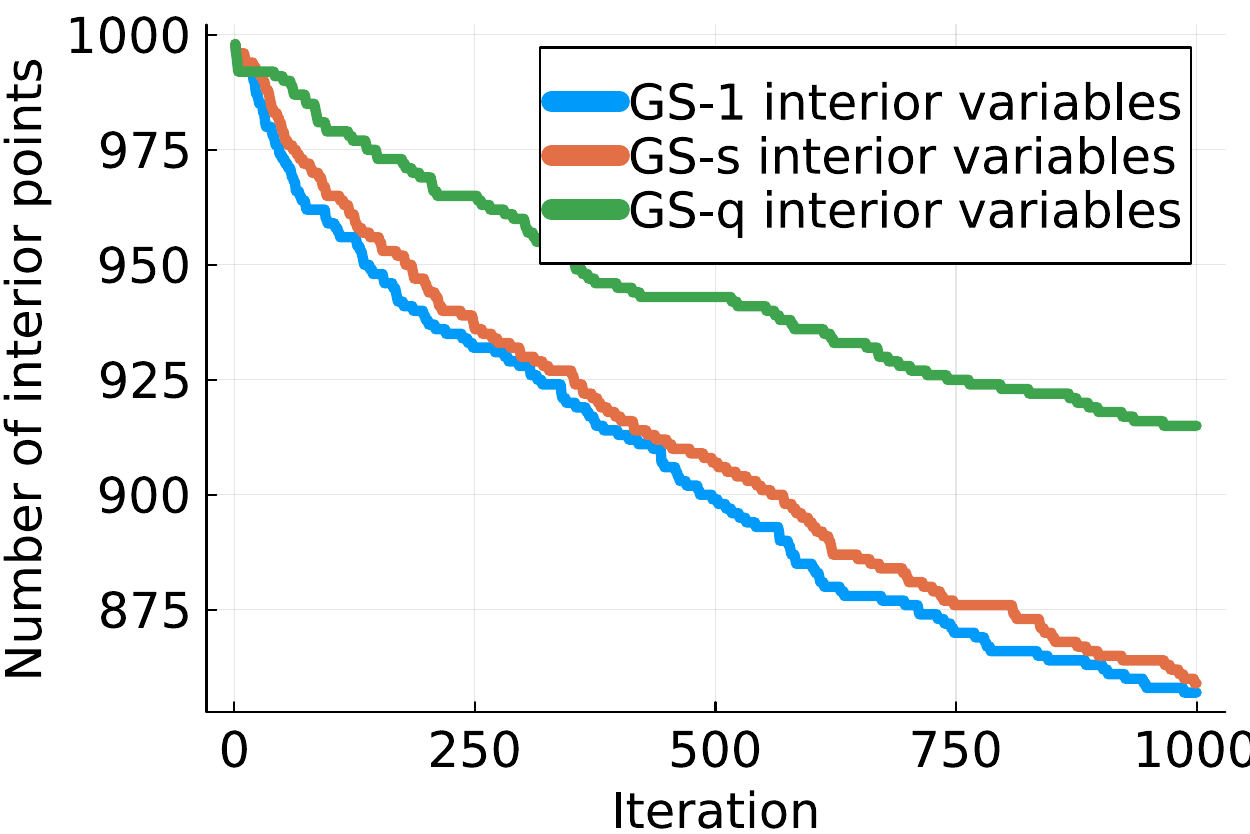}
	\includegraphics[width=0.49\textwidth]{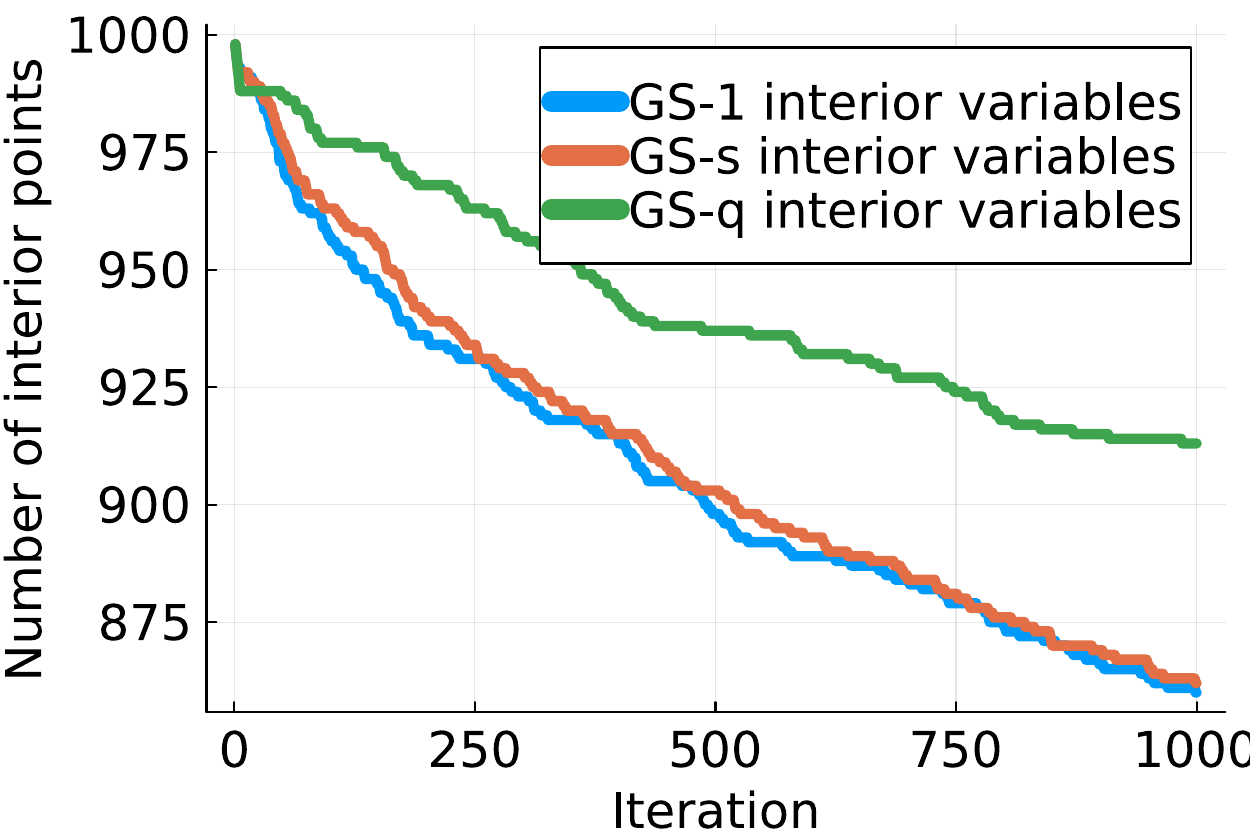}
	\includegraphics[width=0.49\textwidth]{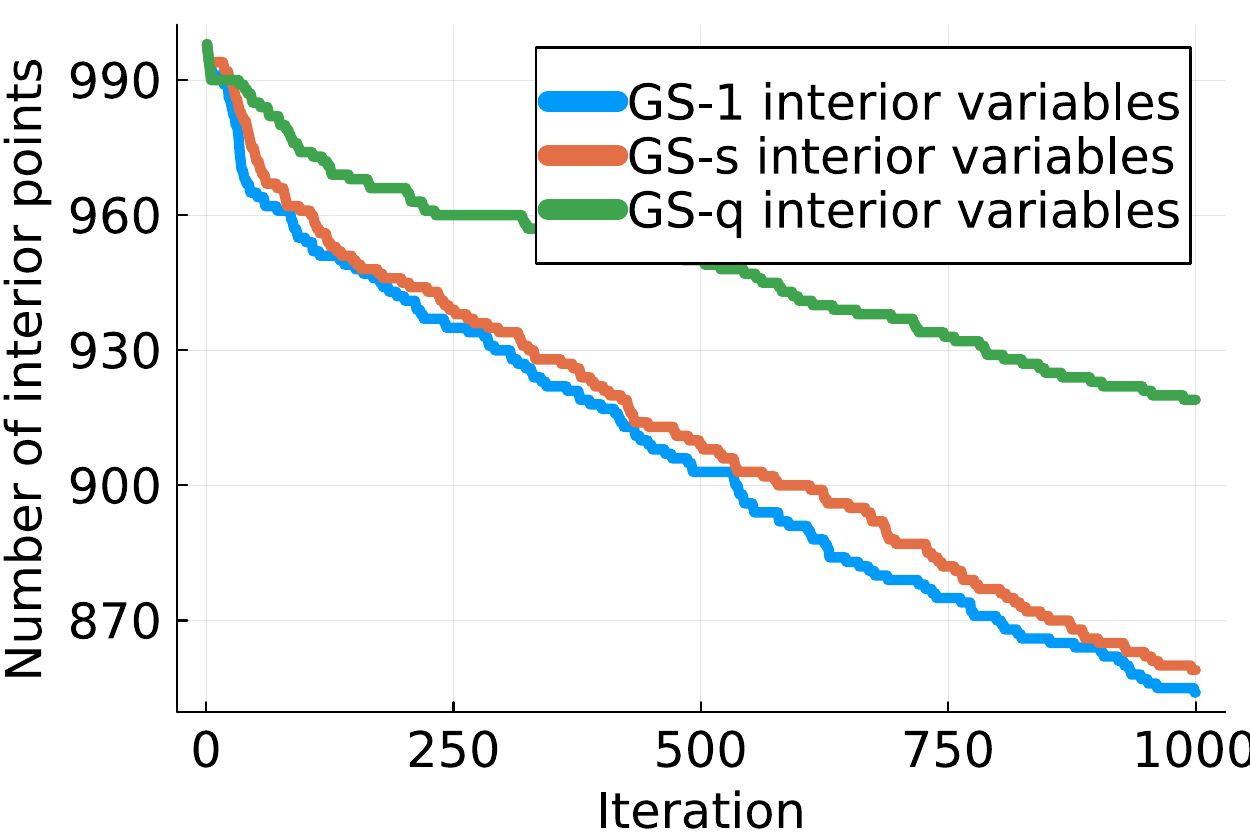}
	\includegraphics[width=0.49\textwidth]{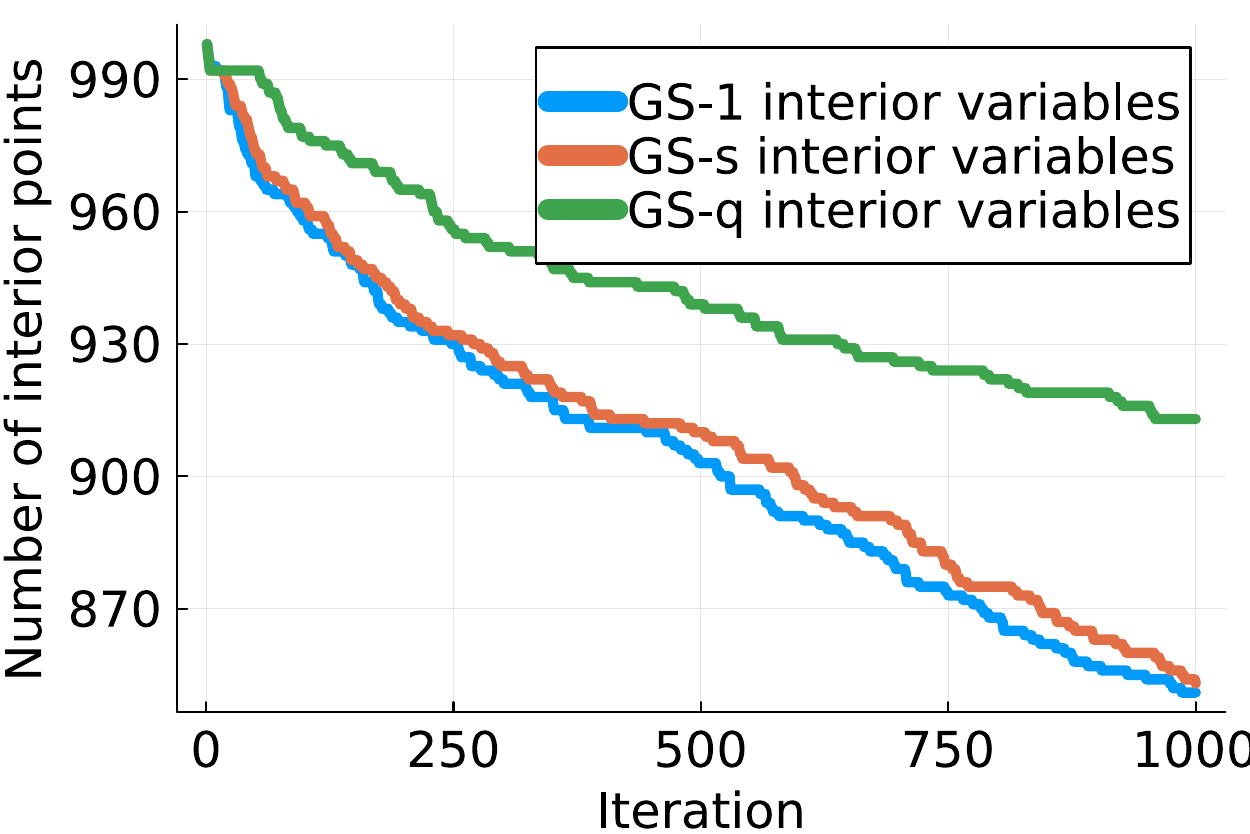}
	\caption{Comparison of number of interior variables updated by GS-1, GS-q and GS-s in every iteration for data generated by different random seed}
	\label{fig:intvardiffseed}
\end{figure}

\begin{figure}
	\centering    \includegraphics[width=0.49\textwidth]{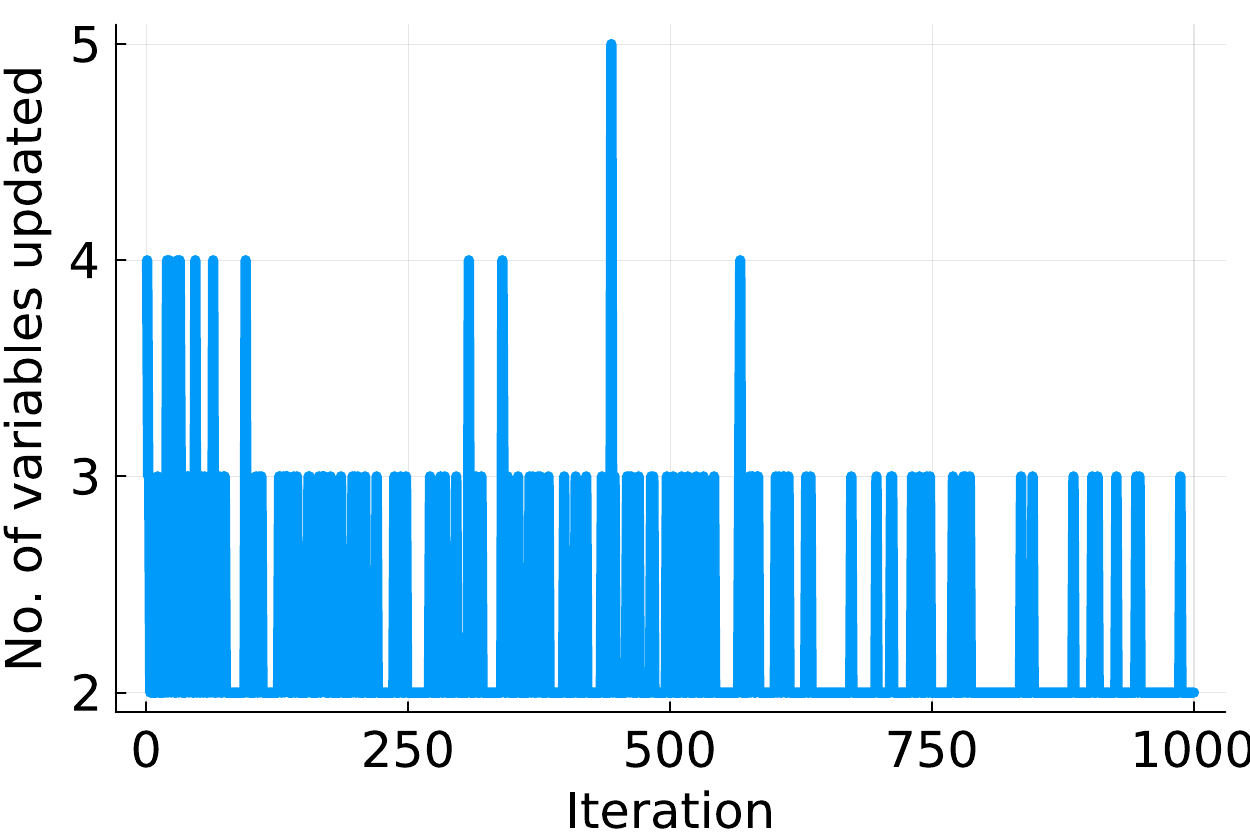}
	\includegraphics[width=0.49\textwidth]{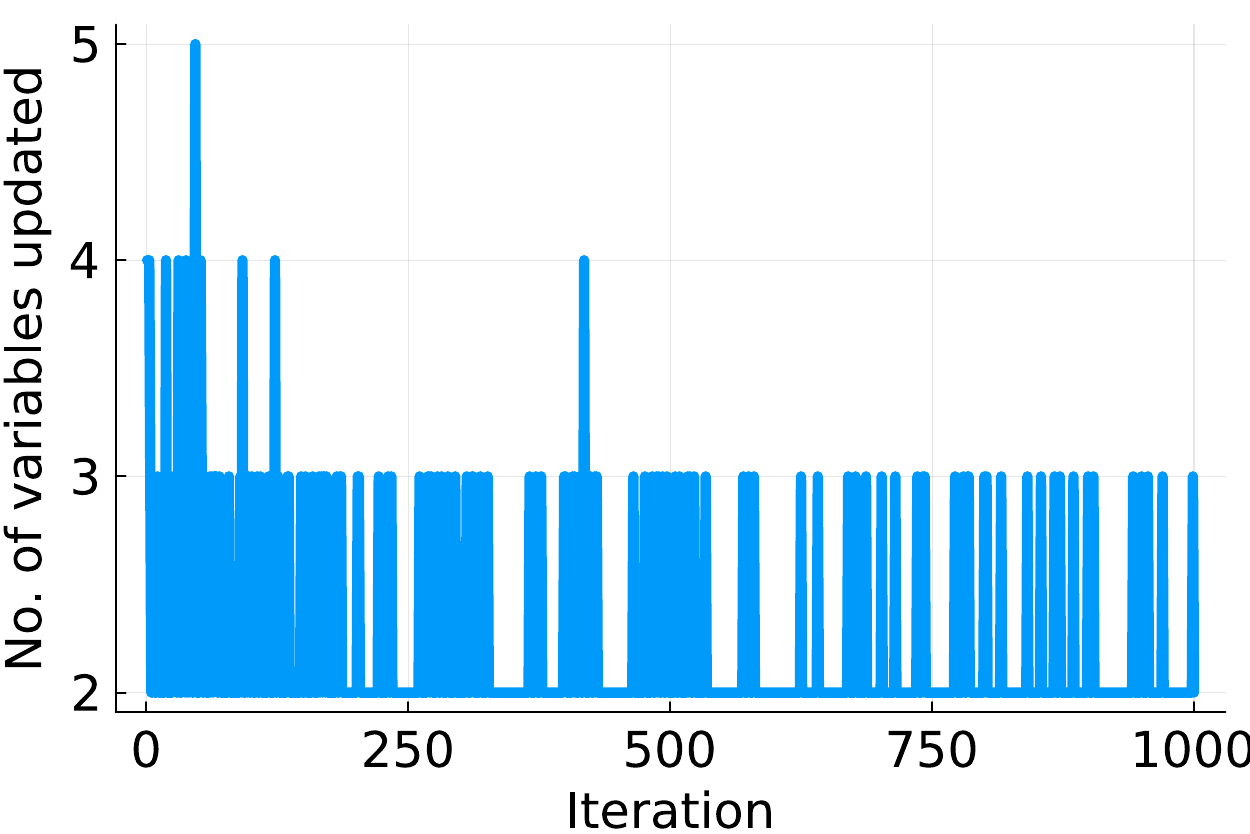}
	\includegraphics[width=0.49\textwidth]{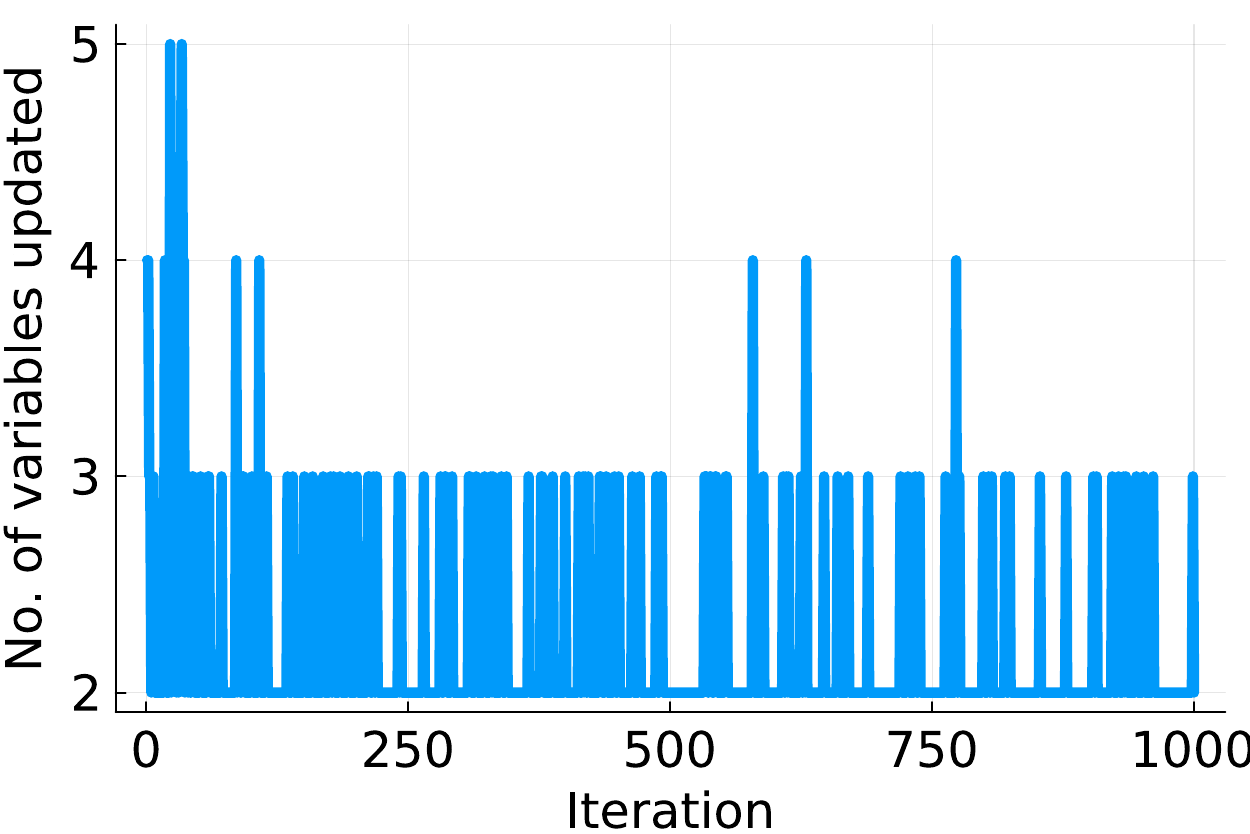}
	\includegraphics[width=0.49\textwidth]{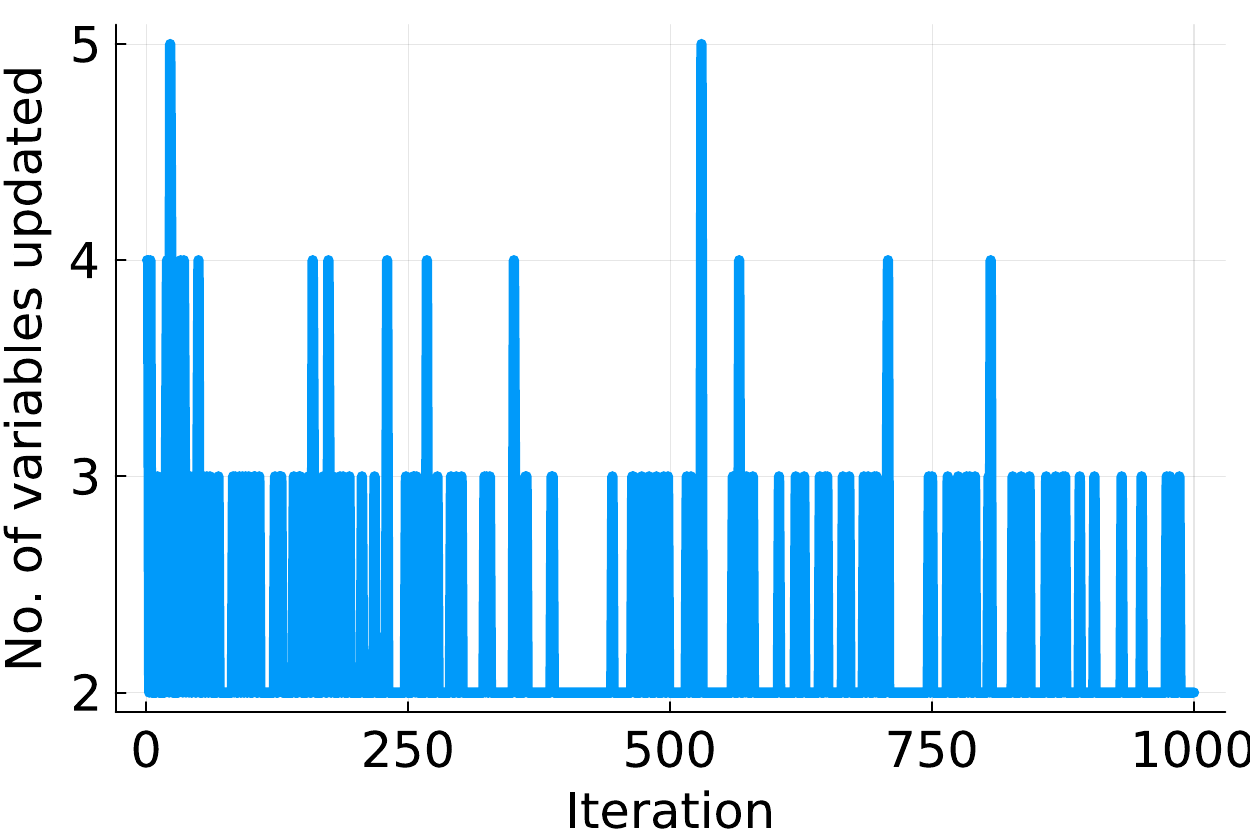}
	\caption{Number of variables updated by GS-1 with different random seed used to generate the data.}
	\label{fig:numvardiffseed}
\end{figure}


\end{document}